\documentclass[10pt]{article}
\usepackage{hyphenat}

\usepackage{cancel} 
\usepackage{xfrac} 
\usepackage{float}
\usepackage{amsfonts,amssymb,amsmath,latexsym,amsthm,mathrsfs,cases} 
\usepackage{tikz}
\usetikzlibrary{arrows,automata,positioning}
\usepackage{graphicx}  
\usepackage[margin=1.05in]{geometry}
\usepackage{arydshln}
\usepackage{etoolbox}
\usepackage{color}
\usepackage{listings} 
\usepackage[final, nopatch=footnote]{microtype} 
\usepackage{enumerate}
\usepackage[shortlabels]{enumitem}
\usepackage{empheq} 
\usepackage{titlesec}
\usepackage{subcaption}
\usepackage{todonotes}
\usepackage{mathrsfs}
\usepackage{upgreek}
\usepackage{comment}
\usepackage{caption}
\usepackage{listings} 
\usepackage{xcolor} 
\usepackage{pdfpages} 
\usepackage[most]{tcolorbox} 
\usepackage{color, soul} 
\usepackage{tikz-cd} 
\usepackage{svg}
\usepackage[numbers]{natbib}

\newtcolorbox{myboxg}{
  breakable,
  colback=blue!5!white,
  colframe=blue!75!black
}

\usepackage[hidelinks]{hyperref}  
\definecolor{darkblue}{rgb}{0.0,0.0,0.3}
\hypersetup{colorlinks,breaklinks,
            linkcolor=blue,urlcolor=blue,
            anchorcolor=darkblue,citecolor=darkblue}

\titleformat{\section}[block]
  {\normalfont\large\bfseries}{Section \thesection:}{5pt}{}
\titleformat{name=\section,numberless}[block]
  {\normalfont\large\bfseries}{}{0pt}{}

\newcounter{homework}

\newcounter{exampractice} 

\theoremstyle{plain}
\newtheorem{theorem}{Theorem}[section]

\newtheorem{proposition}[theorem]{Proposition}
\newtheorem{corollary}[theorem]{Corollary}

\newtheorem{algorithm}[theorem]{Algorithm}

\newtheorem{notation}[theorem]{Notation}

\theoremstyle{definition}

\newtheorem{convention}[theorem]{Convention}

\newtheorem{definition}[theorem]{Definition}
\newtheorem{result}[theorem]{Result}

\theoremstyle{remark}
\newtheorem{remark}[theorem]{Remark}

\newtheorem{example}[theorem]{Example}

\newtheorem{note}{Note}

\newtheoremstyle{exercise-style}
  {\topsep}
  {\topsep}
  {}
  {}
  {\itshape}
  {}
  {.5em}
  {\thmname{#1}\thmnumber{ #2}\thmnote{ (#3)}.}
\theoremstyle{exercise-style}

\numberwithin{equation}{section}
\numberwithin{theorem}{section}

\newcommand{\BN}{\mathbb{N}}

\newcommand{\BR}{\mathbb{R}}

\title{A Stable Distance Persistence Homology for Dynamic Bayesian Network Clustering}

\author{Will Bales,  Carmen Rovi}

\newcommand{\Addresses}{{
\bigskip
\footnotesize

Will Bales,  \textsc{Department of Mathematics \& Statistics, Loyola University Chicago
} \par \nopagebreak
\textit{E-mail address}, W.~Bales:  \texttt{wbales@luc.edu}

\medskip

Carmen Rovi,  \textsc{Department of Mathematics \& Statistics, Loyola University Chicago 
} \par \nopagebreak
\textit{E-mail address}, C.~Rovi:  \texttt{crovi@luc.edu}

}}

\begin{document}

\maketitle

\begin{abstract}
Dynamic Bayesian networks (DBNs) are a widely used framework for modeling
systems whose probabilistic structure evolves over time. Standard inference
methods focus on local conditional distributions and can miss larger-scale
patterns in how dependencies between variables organize and change over time.
We introduce a topological approach to this problem. To each DBN we associate
a time-varying graph, called a Dynamic Bayesian Graph (DBG), by
assigning to each edge a strength that measures variation in its conditional
dependence across parent configurations, and retaining edges whose strength
exceeds a chosen threshold. We show that this construction fits within the
dynamic graph framework of Kim and M\'{e}moli \cite{KimMemoli2021}, enabling the use
of tools from topological data analysis.
Applying persistent homology to a DBG produces a barcode, which records the
merging and disappearance of connected groups of strongly dependent variables
over time. We prove that this barcode is stable: small perturbations in the
conditional probability tables of the DBN lead to small changes in the
resulting barcode. This yields a principled and noise-resistant summary of
how dependency structure evolves in a dynamic Bayesian network.
\end{abstract}

\setcounter{tocdepth}{2}
\tableofcontents
\setcounter{section}{0}    %

\section*{Introduction}

One of the great utilities of probabilistic graphical models is the ability
to take a seemingly complex system and capture its behavior using a small set
of probabilistic dependencies among random variables. Dynamic Bayesian Networks
(DBNs) do this across time: they model systems whose state evolves through
successive copies of a Bayesian network \cite{Koller-Friedman}, with conditional
probability tables (CPTs) that may vary from one time step to the next. They
have found applications across domains as varied as computational biology,
neuroscience, epidemiology, and engineering. A key structural feature of DBNs
is that their underlying directed graph is fixed (only the CPT values change) yet it is precisely these variations that give rise to evolving dependence
structures among the nodes. Understanding and tracking those structures is
central to the analysis of DBNs.

\subsubsection*{Structure on a DBN}

It is widely known that DBNs maintain a directed acyclic graph (DAG) as their
underlying structure. This DAG encodes a compact factorization of the joint
distribution and a precise conditional independence structure. While the
topology of the DAG does not evolve across time, the CPT of each node varies
between time slices, and these variations carry meaningful information about
shifting probabilistic dependencies. To extract this information
geometrically, we introduce a time-indexed edge strength function
$\delta^{(t)}_{YX}$ defined on the directed edges of the DBN. By assigning
each edge a real-valued strength at each time, we can impose geometric
structure on the DBN and bring to bear tools from topological data analysis
(TDA). Specifically, we use a method called Stable Distance Persistent
Homology (SDPH) to extract time-evolving clustering information from a DBN in
a provably stable manner: small perturbations in the CPT values
produce correspondingly small changes in the resulting clustering signatures, making these clustering signatures resistant to noise. 

\subsubsection*{Clustering on a DBN}

A fundamental inference task on a DBN is \emph{filtering}: for each time $t$,
computing the belief state
$\sigma^{(t)} = P(\mathcal{X}^{(t)} \mid \mathbf{E}^{(0:t)})$,
the distribution over the current system state given all evidence observed up
to time $t$. The belief state grows exponentially in $|\mathcal{X}|$, making
exact computation intractable in general. Algorithms for Bayesian network
clustering \cite{wu2025optimal, Albrecht-Ramamoorthy} address this by grouping nodes into clusters
whose local distributions can be computed separately, reducing the effective
problem size. However, most of these methods are heuristics that do not
leverage the inherent topological structure of the underlying DAG, and they
provide no principled framework for tracking how clusters evolve across time
or for comparing the cluster structure of different networks.

Our approach remedies this. We construct a formal framework using SDPH in
order to identify and track key clustering structures in a DBN as they change
across time. The framework is grounded in the TDA machinery of Kim and
M\'{e}moli \cite{KimMemoli2021}, which introduces \emph{dynamic graphs} and
\emph{formigrams} for tracking how connected components (read here as
clusters) are born, merge, and disband as a graph evolves continuously in
time. Applying zigzag persistence to a formigram produces a \emph{barcode},
a multiset of intervals whose lengths record the persistence of individual
clustering events. The SDPH approach is particularly well-suited to DBNs
because it reads the evolving conditional independence structure encoded in the
DAG, rather than treating the network as raw data to be clustered by external
criteria.

\subsubsection*{Main Results}

The central obstacle in applying the framework of \cite{KimMemoli2021} to DBNs is
structural: DBNs carry directed edges and operate in discrete time, while the
dynamic graphs of \cite{KimMemoli2021} are undirected, continuous-time objects required
to satisfy four axioms: self-loops, tameness, interval lifespan, and, most
delicately, comparability. Our primary contribution is the construction of a
\emph{Dynamic Bayesian Graph} (DBG), a dynamic graph derived canonically from
a DBN that resolves this mismatch (Section \ref{Sec:stable-distance}).

The construction rests on the edge strength function $\delta_{ji}$ of Leonelli
and Smith \cite{diameter-LeonelliSmith2025}, defined as the maximum upper diameter of the CPT of a
child node over configurations of its remaining parents
(Definition~\ref{Def:edge-strength}). This quantity has a precise conditional
independence interpretation: $\delta_{ji} = 0$ if and only if
$X_i \perp X_j \mid X_{\Pi_i \setminus j}$
(Remark~\ref{Rmk:indepence-interpretation}), giving the threshold parameter $\eta$ a natural
meaning as a cut-off for non-negligible conditional influence. Symmetrizing
the directed edge set and retaining only edges whose strength exceeds $\eta$
yields a time-varying undirected edge set $E_\eta(\cdot)$. The key technical
step is the \emph{maximum convention} at critical times: at each transition
moment $t^* = k\Delta t$ between time slices, the edge set is assigned the
union of its left and right limits. We prove that this is the unique choice
consistent with the comparability axiom of \cite[Definition~2.3(iv)]{KimMemoli2021},
and that the resulting object
$\mathscr{G}_{\mathcal{X}} = (\mathcal{X}_{\rightarrow}^{(\cdot)}, E_\eta(\cdot))$
is a well-defined dynamic graph (Proposition~\ref{Prop:dynamic-Bayesian-dynamic-graph}). Every
DBG is moreover a \emph{saturated} dynamic graph: its vertex set
$\mathcal{X}$ is constant in time, a structural feature that simplifies the
subsequent analysis.

With the DBG established, the machinery of \cite{KimMemoli2021} applies directly.
We define the formigram $\pi_0(\mathscr{G}_{\mathcal{X}})$ via the path
component functor, the barcode $\mathrm{dgm}(\mathscr{G}_{\mathcal{X}})$
via the map extension $\Psi_K$, and the $\varepsilon$-smoothing
$S_\varepsilon\mathscr{G}_{\mathcal{X}}$ (Section \ref{Prop:dynamic-Bayesian-dynamic-graph}). Our stability result
(Proposition~\ref{Prop:distance}) states that
\[
  d_B\!\bigl(\mathrm{dgm}(S_\varepsilon\pi_0(\mathscr{G}_{\mathcal{X}})),\,
             \mathrm{dgm}(\pi_0(\mathscr{G}_{\mathcal{X}}))\bigr)
  \;\leq\; \varepsilon,
\]
so that smoothing a DBG by $\varepsilon$ changes its barcode by at most
$\varepsilon$ in the bottleneck distance. We also establish that the
interleaving distance $d_1^{\mathrm{dynG}}$ defines an extended
pseudo-metric on DBGs (Theorem~\ref{Thm:pseudometric}), enabling principled
comparison between networks.

Finally, the time-varying path component clustering map and the time-varying
clustering of a DBN (Section~\ref{Sec:stable-distance}) translate the topological output back into
the language of DBNs: each bar in the barcode records a group of random
variables that remain jointly linked throughout the bar's lifespan, with
merging events corresponding to clusters becoming probabilistically linked and
disbanding events corresponding to the breaking of such links.

\smallskip
\subsubsection*{Organization.}
Section~1 summarizes the relevant background from \cite{KimMemoli2021}: dynamic
graphs, formigrams, zigzag persistence, and stability. Section~2 recalls
the standard definitions of Bayesian and dynamic Bayesian networks following
\cite{Koller-Friedman}. Section~3 introduces the edge strength function following
\cite{diameter-LeonelliSmith2025}. Section~4 presents the DBG construction and the main results.
Sections~5--7 place the construction in the context of existing BN clustering
methods, develop the clustering interpretation, and discuss the design
flexibility of the edge strength function.

\section{Stable Distance Persistent Homology}\label{Sec:Kim-Memoli-summary}
The following section summarizes the framework developed by Kim and M\'{e}moli~\cite{KimMemoli2021};
All definitions and results in this section are drawn from that reference unless otherwise noted.

\subsection{Introducing SDPH}

Stable Distance Persistent Homology is a method in TDA that applies zigzag persistent homology to extract topological signatures from a dynamic graph $\mathscr{G}_X$. Concretely, $\mathscr{G}_X$ is encoded as a sequence of zigzag simplicial filtrations, each of whose constituent simplicial complexes is 1-dimensional. Since a 1-dimensional simplicial complex is simply a graph, its only non-trivial homology groups are $H_0$ (connected components) and $H_1$ (independent cycles); for the purpose of clustering, the relevant invariant is $H_0$. Accordingly, this paper focuses on the 0-th homology barcode as a stable invariant of $\mathscr{G}_X$.

The key result from \cite{KimMemoli2021} is that these barcodes are stable invariants. This is shown using what is called the bottleneck distance and the interleaving distance. The bottleneck distance is used specifically for barcodes, whereas the interleaving distance tells us the distance between dynamic graphs (DGs) $\mathscr{G}_X$. Their key result can be summed up in the following theorem.

\begin{theorem} \cite[Thm 1.1]{KimMemoli2021}
    If we consider $\textnormal{dgm}_0(\mathscr{G}_X)$ and $\textnormal{dgm}_0(\mathscr{G}_Y)$ to be the clustering barcodes for the dynamic graphs $\mathscr{G}_X$ and $\mathscr{G}_Y$ over sets $X$ and $Y$ that are finite and non-empty, then we know that
    \begin{align}
        d_B(\textnormal{dgm}_0(\mathscr{G}_X), \textnormal{dgm}_0(\mathscr{G}_Y)) \leq 2 d_1^{\textnormal{dynG}}(\mathscr{G}_X, \mathscr{G}_Y),
    \end{align}
    where $d_1^{\textnormal{dynG}}(\mathscr{G}_X, \mathscr{G}_Y)$ is the interleaving distance between $\mathscr{G}_X$ and $\mathscr{G}_Y$.
\end{theorem}

Furthermore, \cite{KimMemoli2021} creates the notion of a \textit{formigram} which describes the clustering of a DG. This formigram is used to obtain a DG's persistence barcode with a robust algebraic interpretation. 

While there are many different ways to utilize this method for zigzag persistent homology, there are several different ways that this algorithm becomes key to the analysis of DBNs.

\begin{enumerate}[(i)]
    \item The time-varying nature of the zigzag persistence diagram for a DG reflects the progression of its clustering features across time.
    \item Formigrams effectively represent birth and death events in a DG. For this reason, the clustering features represented by a formigram can express points with infinite lifespan and finite lifespan.
    \item A notion of stability is created between the clustering barcodes of DGs and the actual distance between DGs. This makes their results especially desirable for the purposes of data analysis.
    \item Clustering is dependent on a self-prescribed edge strength function which can add richness and variety to the overall interpretation of the clustering barcodes.
\end{enumerate}

Given a dynamic graph, it takes three steps to fully obtain a persistence barcode along with stability results. Step 1 consists of lifting a dynamic graph into a formigram. Step 2 consists of carefully representing a formigram as a zigzag persistence barcode. And step 3 consists of demonstrating the stability of our zigzag persistence barcode.

\subsection{Step 1: Dynamic Graphs to Formigrams}

\subsubsection{Dynamic Graphs}

\begin{definition}\label{Def:Dynamic-graph} \cite[Def 4.1]{KimMemoli2021}(Dynamic Graph)
    Consider the following functions
    \begin{align}
        V_X(\cdot): \textbf{R} \to \mathcal{P}(X) \quad \textnormal{and} \quad E_X(\cdot): \textbf{R} \to \mathcal{P}(\mathcal{P}_2(V_X(\cdot))).
    \end{align}
    A \textbf{dynamic graph} $\mathscr{G}_X$ over $X$ is defined as the pair of maps $(V_X(\cdot),E_X(\cdot))$, satisfying the conditions below:
    \begin{enumerate}[(i)]
        \item (Self-loops) We say if $x \in V_X(t), $ for $t \in \mathbb{R}$, then $(x,x) \in E_X(t)$.
        \item (Tameness) We say $\textnormal{crit}(\mathscr{G}_\mathcal{X})$ is locally finite.
        \item (nodes' Lifespan) We say that for a node $x \in X$, the following set is non-empty, closed interval $I_x \coloneq \{ t \in \mathbb{R} | x \in V_X(t) \}$. This set is called the $\textbf{lifespan}$  of $x$.
        \item (Comparability) We say that the following inclusions hold
        \begin{align}
            V_X(t - \varepsilon) \subset V_X(t) \supset V_X(t + \varepsilon) \quad \textnormal{and} \quad E_X(t- \varepsilon) \subset E_X(t) \supset E_X(t+ \varepsilon)
        \end{align}
        for $t \in \mathbb{R}$ and $\varepsilon>0$ that is small enough to allow the condition to hold. This is rewritten as $\mathscr{G}_X(t- \varepsilon) \subset \mathscr{G}(t) \supset \mathscr{G}_X(t + \varepsilon)$ for brevity.
    \end{enumerate}
\end{definition}

We can observe that the comparability conditions force $\mathscr{G}(c)$ to be locally maximum. In other words, at each critical point $c\in\mathrm{crit}(\mathscr{G}_X)$, at least one of the inclusions $\mathscr{G}_X(c-\varepsilon)\subset\mathscr{G}_X(c)\supset\mathscr{G}_X(c+\varepsilon)$ is strict.

Furthermore, we can also define the notion of an isomorphism between Dynamic graphs.

\begin{definition}[Isomorphisms between DGs]
    If we consider two DGs $\mathscr{G}_X = (V_X(\cdot),E_X(\cdot))$ and $\mathscr{G}_Y = (V_Y(\cdot), E_Y(\cdot))$, we define an isomorphism between these two graphs as the bijection $\varphi : X \to Y$ where for $t \in \mathbb{R}$
    \begin{enumerate}[(i)]
        \item we have $V_X(t) = \varphi(V_Y(t))$,
        \item and  $(x,x') \in E_X(t) \iff (\varphi(x), \varphi(x)') \in E_Y(t).$
    \end{enumerate}
\end{definition}

\subsubsection{Formigrams}

Formigrams live in the category $\textbf{Sets}$. They are in fact  collections of sub-partitions of our set $X$ that have the potential to be realized through time. First we will define what a collection sub-partition is.

\begin{definition}[Collection of Sub-partitions]
    If we consider some non-empty and finite set $X$, we can define the \textbf{collection of sub-partitions of} $X$:
    \begin{enumerate}[(i)]
        \item The set
        \begin{align}
            \mathcal{P}^{\textnormal{sub}}(X) \coloneq \{ P ~|~ \exists X' \subset X, P \textnormal{ is a partition of } X'\}
        \end{align}
        is called the collection of all sub-partitions of $X$.
        \item We call the set $\mathcal{P}(X)$ the sub-collection $\mathcal{P}^{\textnormal{sub}}(X)$ that contains only partitions of the entire set $X$.
    \end{enumerate}
\end{definition}

\begin{definition}\cite[Def 5.4]{KimMemoli2021}\label{Def:formigram}(Formigram)
    We can define the \textbf{formigram} over a set $X$ as the function $\theta_X : \mathbb{R} \to \mathcal{P}^{\textnormal{sub}}(X)$ where we have a tameness, interval lifespan, and comparability condition:
    \begin{enumerate}[(i)]
        \item (Tameness) If we let $\textnormal{crit}(\theta_X)$ represent the set of critical points for the function $\theta_X$ where this set is locally finite, we satisfy the tameness condition.
        \item (Interval Lifespan) Similar to a DG, we say that for an element $x \in X$, the following set is non-empty and closed $I_X \coloneq \{ t \in \mathbb{R} | x \in B \in \theta_X(t) \}$. This set is called the $\textbf{lifespan}$  of $x$.
        \item (Comparability) The condition $\theta_X(c - \varepsilon) \leq \theta_X(c) \geq \theta_X(c+ \varepsilon)$ for $c \in \mathbb{R}$ and small $\varepsilon>0$. This means $\theta_X(c)$ is locally maximally coarse.
    \end{enumerate}
\end{definition}

Formigrams allow for the analysis of merging, disbanding, birth, and death structures in a given time-evolving subpartition.

\begin{definition}[The Merging, Disbanding, Birth, and Death Events of A Formigram]
    If we have a formigram $\theta_X$ over a set $X$, and we let $t_0 \in \mathbb{R}$ be in $\textnormal{crit}(\theta_X)$ with $\varepsilon > 0$ and $\{ t_0 \} = \{ t_0 - \varepsilon, t_0 + \varepsilon \} \cap \textnormal{crit}(\theta_X)$ for sufficiently small $\varepsilon$, then we can say the following:
    \begin{enumerate}[(i)]
        \item A \textit{merging event} happens at $t_0 \in \mathbb{R}$ if for two non-intersecting blocks $A,B \in \theta_X(t_0 - \varepsilon)$, we have $C \in \theta_X(t_0)$ with $A \cup B \subset C$.
        \item A \textit{disbanding event} happens at $t_0 \in \mathbb{R}$ if for two non-intersecting blocks $A,B \in \theta_X(t_0 + \varepsilon)$ we have $C \in \theta_X(t_0)$ with $A \cup B \subset C$.
        \item A \textit{birth event} occurs at  time $t_0 \in \mathbb{R}$ if we have $x \in X$ that exists in one of the sub-partitions of $\theta_X(t_0)$, but is not in one of those sub-partitions in $\theta_X(t_0 - \varepsilon)$.
        \item A \textit{death event} occurs at time $t_0 \in \mathbb{R}$ if we have an $x \in X$ that exists in one of the sub-partitions of $\theta_X(t_0)$, but is no longer one of those sub-partitions in $\theta_X(t_0 + \varepsilon)$.
    \end{enumerate}
\end{definition}

\subsubsection{Dynamic Graphs to Formigrams}

Before DGs can be lifted into the world of formigrams, we must understand what important information will be represented of our graph in the category $\textbf{Sets}$. For this reason \cite{KimMemoli2021} defines the \textit{path component functor} $\pi_0$.

\begin{definition}
   [Path Component Functor] The path component functor is the functor $\pi_0 : \mathbf{Graph} \to \mathbf{Sets}$ defined as follows:
\begin{enumerate}[(i)]
    \item 
\textbf{On objects:} For a graph $G_X = (X, E_X)$, declare $x \sim x'$ if there exists a finite sequence $x = x_1, x_2, \ldots, x_n = x'$ in $X$ such that $(x_i, x_{i+1}) \in E_X$ for each $i = 1, \ldots, n-1$. The relation $\sim$ is an equivalence relation on $X$, and we define
$$\pi_0(G_X) \coloneqq X/{\sim},$$
to be the partition of $X$ into path-connected components.

\item \textbf{On morphisms:} A graph morphism $f : G_X \to G_Y$ induces a map $\pi_0(f) : \pi_0(G_X) \to \pi_0(G_Y)$ that sends each block $B \in \pi_0(G_X)$ to the unique block $C \in \pi_0(G_Y)$ satisfying $f(B) \subseteq C$. 
\end{enumerate}
\end{definition}

With this path component functor, we are able to lift a DG into a formigram.

\begin{definition} \label{Def:formigram-of-DG}
[The Formigram of a DG] Given a DG $\mathscr{G}_X = (V_X(\cdot), E_X(\cdot))$, define the function $\pi_0(\mathscr{G}_X) : \mathbb{R} \to \mathcal{P}^{\mathrm{sub}}(X)$
by
\[\pi_0(\mathscr{G}_X)(t) := \pi_0(\mathscr{G}_X(t)), \qquad \text{for } t \in \mathbb{R}.\]

\end{definition}

\begin{proposition}\label{Prop:DG-to-formigram}
The function $\pi_0(\mathscr{G}_X)$ from Definition \ref{Def:formigram-of-DG} is a formigram. In particular, every DG gives rise to a formigram via the path component functor.
\end{proposition}

\begin{proof}
We verify the three conditions of Definition \ref{Def:formigram}.

\begin{enumerate}[(i)]
    \item \textbf{Tameness:} The partition $\pi_0(\mathscr{G}_X(t))$ can only change at $t$ if the graph $\mathscr{G}_X(t)$ itself changes, so $\mathrm{crit}(\pi_0(\mathscr{G}_X)) \subseteq \mathrm{crit}(\mathscr{G}_X)$. Since $\mathrm{crit}(\mathscr{G}_X)$ is locally finite by DG condition (ii) in Definition \ref{Def:Dynamic-graph}, so is $\mathrm{crit}(\pi_0(\mathscr{G}_X))$.

\item \textbf{ Interval Lifespan:} For $x \in X$, we have $x \in B$ for some $B \in \pi_0(\mathscr{G}_X(t))$ if and only if $x \in V_X(t)$, since the path component functor partitions exactly the vertices present at time $t$. Therefore the lifespan of $x$ in $\pi_0(\mathscr{G}_X)$ is ${t \in \mathbb{R} \mid x \in V_X(t)} = I_x$, which equals the lifespan of $x$ in the DG. This is non-empty by the DG condition in Definition \ref{Def:Dynamic-graph} (iii).

\item \textbf{Comparability.} Fix $c \in \mathbb{R}$. By DG condition (iv), for small enough $\varepsilon > 0$,
$$V_X(c \pm \varepsilon) \subset V_X(c) \quad \text{and} \quad E_X(c \pm \varepsilon) \subset E_X(c).$$
Thus $\mathscr{G}_X(c \pm \varepsilon)$ is a subgraph of $\mathscr{G}_X(c)$. Any path connecting $x$ to $y$ in $\mathscr{G}_X(c \pm \varepsilon)$ is also a path in $\mathscr{G}_X(c)$, so every block of $\pi_0(\mathscr{G}_X(c \pm \varepsilon))$ is contained in a single block of $\pi_0(\mathscr{G}_X(c))$. Hence $\pi_0(\mathscr{G}_X(c \pm \varepsilon)) \leq \pi_0(\mathscr{G}_X(c))$, as required.    \raggedbottom
\end{enumerate}
\end{proof}

\subsection{Step 2: Zigzag Modules to Barcodes}\label{Sec:Zigzag}

\subsubsection{Important Posets}

\begin{definition}[The Poset $\mathbb{R}^{\mathrm{op}} \times \mathbb{R}$]
Let $\mathbb{R}^{\mathrm{op}}$ denote the real numbers equipped with
the opposite order: for $a, b \in \mathbb{R}$,
\[
  a \leq_{\mathrm{op}} b \quad\Longleftrightarrow\quad a \geq b.
\]
The product poset $\mathbb{R}^{\mathrm{op}} \times \mathbb{R}$ is
then ordered componentwise: for
$(a_1, a_2), (b_1, b_2) \in \mathbb{R}^{\mathrm{op}} \times \mathbb{R}$,
\[
  (a_1, a_2) \leq (b_1, b_2)
  \quad\Longleftrightarrow\quad
  a_1 \geq b_1 \;\text{ and }\; a_2 \leq b_2.
\]
\end{definition}

\begin{definition}
    Furthermore, we will define the poset $\mathbb{ZZ}$ as a sub-poset of $\mathbb{R}^{\textnormal{op}} \times \mathbb{R}$ defined by $\mathbb{ZZ} \coloneq \{ (k,l): k \in \mathbb{Z}, l \in \{ k, k-1 \} \}$.
\end{definition}

\subsubsection{Creating and Decomposing Zigzag Modules}

\begin{definition}[Internal Maps]
    Let us consider two different categories. We will have the first category be a poset, $\textbf{P}$, and the second category be an arbitrary category $\textbf{C}$. We will take notice of the functor between these categories, $F: \textbf{P} \to \textbf{C}$. If we have the elements $s, t \in \textbf{P}$, where $F_s \coloneq F(s)$, we can define the morphism in $\textbf{C}$, $F(s \leq t)$, as $\varphi_F(s,t): F_s \mapsto F_t$. For any morphism $s \leq t$, the morphism $\varphi_F(s,t): F_s \to F_t$ is called the \textbf{internal map} of $F$.
\end{definition}

\begin{definition}[$\textbf{P}-$Indexed Module]
    If we take a functor $F: \textbf{P} \to \textbf{Vec}$, where the category $\textbf{Vec}$ contains finite-dimensional vector spaces over $\mathbb{F}$ with linear maps as morphisms. 
\end{definition}

\begin{definition}[Intervals on a Poset]
    The interval $\mathscr{J}$ on a poset $\textbf{P}$ is a subset $\mathscr{J} \subset \textbf{P}$ with the following properties:
    \begin{enumerate}[(i)]
        \item $\mathscr{J} \neq \emptyset$,
        \item if $r,t \in \mathscr{J}$, $s \in \textbf{P}$, and $r \leq s \leq t$, then this implies that $s \in \mathscr{J}$,
        \item if we take $s,t \in \mathscr{J}$, we know there is a sequence $s = u_0, u_1, \cdots , u_{l-1}, u_l = t$ for which 
        $u_i$ and $u_{i+1}$ are comparable in $P$ for each $0 \leq i \leq l-1$. 
        
    \end{enumerate}
\end{definition}

\begin{definition}[Interleval Module]
    We define the \textbf{interleval module}, $I^{\mathscr{J}}: \textbf{P} \to \textbf{Vec}$, for any interval $\mathscr{J}\subset \textbf{P}$ using the following conditions
    \begin{align}
        I_t^{\mathscr{J}} = \left\{ \begin{array}{cc}
           \mathbb{F}  & \textnormal{ if } t \in \mathscr{J} \\
            0 & \textnormal{otherwise}
        \end{array} \right.
        \quad
        \varphi_{I^{\mathscr{J}}}(s,t) = \left \{ \begin{array}{cc}
           \textnormal{id}_\mathbb{F}  & \textnormal{if }s,t \in \mathscr{J}, s\leq t \\
           0  & \textnormal{otherwise}
        \end{array} \right .
    \end{align}
\end{definition}

\begin{definition}[The Direct Sum of $\textbf{P}-$Indexed Modules]
    The direct sum of $\textbf{P}-$indexed modules $F$, and $G$, denoted $F \oplus G$, can be defined in the following manner:
    \begin{enumerate}[(i)]
        \item if we know $t \in \textbf{P}$, then $(F \oplus G)_t \coloneq F_t \oplus G_t$,
        \item if we have $s \leq t$ as a morphism in $\textbf{P}$, we denote the linear map in $\textbf{Vec}$ as $\varphi_{F \oplus G}(s,t): (F \oplus G)_s \to (F \oplus G)_t$, and define it by $\varphi_{F \oplus G}(s,t)(v,w)\coloneq (\varphi_F(s,t)(v),\varphi_G(s,t)(w))$ for any $(v,w) \in (F \oplus G)_s$.
    \end{enumerate}
\end{definition}

\begin{definition}[Decomposable $\textbf{P}-$Indexed Modules]
    We say the a $\textbf{P}-$indexed module $F$ is decomposable if it is naturally isomorphic to $Q \oplus G$ where $Q$ and $G$ are non-trivial $\textbf{P}-$indexed modules. This decomposition is denoted as $F \cong Q \oplus G$. The module $F$ is indecomposable if the only decompositions $F \cong Q \oplus G$ have $Q$ or $G$ equal to the zero module. 
\end{definition}

\begin{definition}[Zigzag Modules]
    \textbf{Zigzag modules} are exactly $\mathbb{ZZ}-$indexed modules $F: \mathbb{ZZ} \to \textbf{Vec}$, where if we have $s \in \mathbb{ZZ}$, $F_s$ is a finite dimensional vector space.
\end{definition}

\begin{remark}[Describing The Barcodes of Zigzag Modules]
    If we have the symbol $<$ represent the partial order on $\mathbb{Z^2}$, we may observe that the intervals of $\mathbb{ZZ}$ are only ever of the four following types. These four types correspond to whether the endpoints are critical points $c_i$ or intermediate points $s_i$:
    \begin{align}
        (b,d)_{\mathbb{ZZ}} &\coloneq \{ (i,j) \in \mathbb{ZZ} : (b,b) < (i,j) < (d,d) \} && \textnormal{for } b<d \in \mathbb{Z} \cup \{ -\infty , \infty \}, \\
        [b,d)_{\mathbb{ZZ}} &\coloneq \{ (i,j) \in \mathbb{ZZ} : (b,b) \leq (i,j) < (d,d) \} && \textnormal{for } b<d \in \mathbb{Z} \cup \{\infty \}, \\
        (b,d]_{\mathbb{ZZ}} &\coloneq \{ (i,j) \in \mathbb{ZZ} : (b,b) < (i,j) \leq (d,d) \} && \textnormal{for } b<d \in \mathbb{Z} \cup \{ -\infty \}
        \\ 
        [b,d]_{\mathbb{ZZ}} &\coloneq \{ (i,j) \in \mathbb{ZZ} : (b,b) \leq (i,j) \leq (d,d) \} && \textnormal{for } b<d \in \mathbb{Z}
    \end{align}
    Moreover, we say $\langle b,d \rangle_{\mathbb{ZZ}}$ denote any of the above without specified type. 
\end{remark}

\begin{theorem}[The Decomposability of Zigzag Modules]
    If we let $F: \mathbb{ZZ} \to \textbf{Vec}$ be a zigzag module, it can be decomposed in the following manner:
    \begin{align}
        F \cong \bigoplus_{j \in J} I^{\langle b_j, d_j \rangle_{\mathbb{ZZ}}}
    \end{align}
    where $J$ is some unique index set (up to permutation of terms in the direct sum).
\end{theorem}

\begin{definition}[The Barcode of A Zigzag Module]
    We define the \textbf{barcode of a zigzag module} as $\textnormal{dgm}(F) = \{ \langle b_j, d_j \rangle_{\mathbb{ZZ}}: j \in J \}$ of $F$.
\end{definition}

\subsection{Step 3: Formigrams to Barcodes}

\subsubsection{Important Definitions}

\begin{definition}[Canonical Map]
    If we have some $P, Q \in \mathcal{P}^{\textnormal{sub}}(X)$ where we have that $P \leq Q,$ we say that the canonical map $P \to Q$ is described as for each block $A \in P$ is sent to a unique block $B \in Q$, where we have that $A \subset B$.
\end{definition}

\begin{definition}[Free Functor]
    We will denote the \textbf{free functor} as the functor $\mathfrak{V}_{\mathbb{F}}: \textbf{Sets} \to \textbf{Vec}$. More specifically, if we are given a set $S$, we define the following 
    \begin{align}
        \mathfrak{V}_{\mathbb{F}}(S) \coloneq \left \{ \sum_{s_i \in S} a_i s_i : a_i \in \mathbb{F}, s_i \in S \right  \}
    \end{align}
    Essentially, this set consists of formal linear combinations of finite terms of elements of the set $S$ over the field $\mathbb{F}$. Furthermore, if we are given a map of sets $f: S \to T$, we say that $\mathfrak{V}_{\mathbb{F}}(f)$ is really the linear map between $\mathfrak{V}_{\mathbb{F}}(S)$ and $\mathfrak{V}_{\mathbb{F}}(T)$, that is created by linearly extending $f$.
\end{definition}

\subsubsection{Formigrams to Zigzag Modules}

\begin{algorithm}\cite[Section~5.1]{KimMemoli2021}(Creating a Zigzag Module from a Formigram)
    The algorithm to create a zigzag module from a formigram can be outlined in a sequence of steps.
    \begin{enumerate}[(i)]
        \item Consider the formigram $\theta_X$ over the set $X$. Since $\theta_X$ changes only at the critical points in $crit(\theta_X)$, it suffices to record its behavior on a discrete subset $K$ of $\mathbb{R}$ that is dense enough to separate consecutive critical values. \\ We will define the collection $\mathfrak{C}(\theta_X)$. For all $C \in \mathfrak{C}(X)$, we have $C= \{ c_i \in \mathbb{R}: i \in \mathbb{Z} \}$ such that $\cdots < c_{i-1}<c_i< c_{i+1}< \cdots, \lim_{i \to +\infty} c_i = +\infty, \lim_{-\infty} c_i=- \infty$, and $\textnormal{crit}(\theta_X) \subset C$. Furthermore, we will define subcollection $\textnormal{Subdiv}(C)$ consisting of sets $S = \{ s_i \in \mathbb{R} :i \in \mathbb{Z} \}$ for each $C \in \mathfrak{C}(\theta_X)$, such that $s_i \in (c_i, c_{i+1}), \forall i \in \mathbb{Z}$. For example, if we choose a $C \in \mathfrak{C}(\theta_X)$ and $X \in \textnormal{Subdiv}(C)$, we can get the set $K$ organized in the following fashion:
        \begin{align}
            K \coloneq C \cup S = \{ \cdots < c_{i - 1} < s_{i-1} < c_i < s_i < c_{i+1} < \cdots \}
        \end{align}
        \item The partial order on $\mathcal{P}^{\mathrm{sub}}(X)$ induced
by coarsening gives a canonical map $\theta_X(s) \to \theta_X(c)$
whenever $s$ lies between two consecutive critical points. The
sub-partition at an intermediate value maps into the coarser
partition at the adjacent critical value. Any $K$ constructed in
this way is called an \emph{indexing set} for $\theta_X$, and the
following chain of canonical maps arises from this set:
        \[
        \begin{tikzcd}
        	& {\theta_X(c_{i-1})} && {\theta_X(c_{i})} && {\theta_X(c_{i+1})} & \\
        	\cdots && {\theta_X(s_{i-1})} && {\theta_X(s_{i})} && \cdots
        	\arrow[from=2-1, to=1-2]
        	\arrow[from=2-3, to=1-2]
        	\arrow[from=2-3, to=1-4]
        	\arrow[from=2-5, to=1-4]
        	\arrow[from=2-5, to=1-6]
        	\arrow[from=2-7, to=1-6]
        \end{tikzcd}
        \]
        \item To extract algebraic invariants, we pass from the category
$\mathbf{Sets}$ to the category $\mathbf{Vec}$ by applying the free
functor $\mathfrak{V}_{\mathbb{F}} : \mathbf{Sets} \to \mathbf{Vec}$,
which sends each sub-partition to the vector space it spans and
extends each canonical map linearly. This yields the following
diagram in $\mathbf{Vec}$:

        \[
            \begin{tikzcd}
            	&& {V_{c_{i-1}}} && {V_{c_{i}}} && {V_{c_{i+1}}} & \\
            	{V_{\theta_X}:} & \cdots && {V_{s_{i-1}}} && {V_{s_{i}}} && \cdots
            	\arrow[from=2-2, to=1-3]
            	\arrow[from=2-4, to=1-3]
            	\arrow[from=2-4, to=1-5]
            	\arrow[from=2-6, to=1-5]
            	\arrow[from=2-6, to=1-7]
            	\arrow[from=2-8, to=1-7]
            \end{tikzcd}
         \]
         These maps are created by taking the indexing set $K$ and mapping them with a $\mathbb{ZZ}$ poset by the following bijection
         \begin{align}
             f_K : K \to \mathbb{ZZ} \textnormal{ where } c_i \mapsto (i,i) \textnormal{ and } s_i \mapsto (i + 1, i) \textnormal{ with } i \in \mathbb{Z}
         \end{align}
         With this bijection, We say that $V_{\theta_X}$ is a $\mathbb{ZZ}-$indexed module.
         \item Since $V_{\theta_X}$ is considered to be a zigzag module, it can be decomposed into a direct sum of interleval modules. This decomposition is defined as such:
         \begin{align}
                V_{\theta_X} \cong \bigoplus_{j \in J} I^{\langle b_j, d_j \rangle_\textbf{ZZ}}.
        \end{align}
        We say that $J$ is some indexing set. We also say that for $j \in J$, $\langle b_j, d_j \rangle_{\mathbb{ZZ}} \in \mathbb{ZZ}$ has an interval in $K$ defined as $f_K^{-1}(\langle b_j , d_j \rangle_{\mathbb{ZZ}})$ where this interval contains consecutive elements of $K$.
    \end{enumerate}
\end{algorithm}

\begin{definition}[Map Extension to Intervals in $\mathbb{R}$]
    Take the sets $K$ and $\mathbb{R}$. We define $\textbf{Int}(K)$ and $\textbf{Int}(\mathbb{R})$ as the set of all intervals in each particular set. Using these sets we can create the map extension
    \begin{align}
        \Psi_K : \textbf{Int}(K) \to \textbf{Int}(\mathbb{R})
    \end{align}
    defined in the following way:
    \[
    \begin{array}{c|c}
    \mathrm{Int}(K) & \mathrm{Int}(R) \\ \hline
    [c_i, c_j]_K \;\mapsto\; [c_i, c_j] & \text{for } i \leq j \\[6pt]
    [c_i, s_j]_K \;\mapsto\; [c_i, c_{j+1}) & \text{for } i \leq j \\[6pt]
    [s_i, c_j]_K \;\mapsto\; (c_i, c_j] & \text{for } i < j \\[6pt]
    [s_i, s_j]_K \;\mapsto\; (c_i, c_{j+1}) & \text{for } i \leq j \\[10pt]
    (-\infty,\infty)_K \;\mapsto\; (-\infty,\infty) & \\[6pt]
    (-\infty,c_i]_K \;\mapsto\; (-\infty,c_i] & \\[6pt]
    (-\infty,s_i]_K \;\mapsto\; (-\infty,c_{i+1}] & \\[6pt]
    [c_i,\infty)_K \;\mapsto\; [c_i,\infty) & \\[6pt]
    [s_i,\infty)_K \;\mapsto\; (c_i,\infty) &
    \end{array}
    \]
    In the last entry, note that because $s_i > c_i$ strictly, the corresponding real interval is open at $c_i$. 
\end{definition}

\subsubsection{Barcodes of Formigrams}

\begin{remark}
    Using the map extension $\Psi_K$ we arrive at a multiset containing intervals of $\mathbb{R}$ defined by $\Psi_K \circ \textnormal{dgm}_K(V_{\theta_X}) \coloneq \{ \{ \Psi_K (f^{-1}(\langle b_j, d_j \rangle_{\mathbb{ZZ}})): j \in J \} \}$,  where $\{\{ ~ \}\}$ denotes a multiset, allowing repeated elements. 
\end{remark}

\begin{definition}\label{Def:Barcode-formigram}[Barcode of A Formigram]
    The \textbf{barcode of a Formigram} $\theta_X$ over the set $X$ is the multiset $\Psi_K \circ \textnormal{dgm}_K(\mathbb{V}_{\theta_X})$. This barcode is denoted as $\textnormal{dgm}(\theta_X)$.
\end{definition}

\begin{note}
    We will note that there is invariance associated with the indexing set $K = C \cup S$. As long as $C \in \mathfrak{C}(\theta_X)$ and $S \in \textnormal{Subdiv}(C)$, the multiset $ \Psi_k \circ \textnormal{dgm}_K(\mathbb{V}_{\theta_X})$ is not affected by the set $K$. This irrelevance of a unique specification of $K$ is further explained in \cite{KimMemoli2021}.
\end{note}

\subsection{Stability Results}\label{Sec:stability}

\begin{notation}
    We must first define the notion of a distance between DGs. We will first create the following notation.
    \begin{enumerate}[(i)]
        \item We say that $A + t \coloneq \{ a +t : a \in A \}$ for any subset $A \subset \mathbb{R}$ and any particular $t \in \mathbb{R}$.
        \item We say that $A^\varepsilon \coloneq [ a - \varepsilon, b + \varepsilon ]$ for a closed interval $A = [a,b] \subset \mathbb{R}$, and $\varepsilon \geq 0$.
        \item We say $[t]^\varepsilon \coloneq [t - \varepsilon, t + \varepsilon] \subset \mathbb{R}$ for any $t \in \mathbb{R}$ and $\varepsilon \geq 0$.
     \end{enumerate}
\end{notation}

\subsubsection{Creating $\varepsilon-$smoothing}

\begin{remark}
    Now the notion of $\varepsilon-$smoothing for $\varepsilon \geq 0$ will be introduced as an essential component in defining the distance between DGs. This $\varepsilon-$smoothing allows us to smooth our DGs so that they are resistant to noise and other short lived phenomena.
\end{remark}

\begin{definition}[The Time-Interlevel Smoothing of A DG]
    Consider the DG $\mathscr{G}=(V_X(\cdot), E_X(\cdot))$. We will define the following
    \begin{enumerate}[(i)]
        \item Let $I \subset \mathbb{R}$, then we can create the following:
        \begin{align}
            \bigcup_I \mathscr{G}_X \coloneq \left ( \bigcup_{t \in I} V_X(t), \bigcup_{t \in I} E_X(t)  \right ).
        \end{align}
        \item We say the $\varepsilon-$smoothing $S_\varepsilon \mathscr{G}_X$ of $\mathscr{G}_X$ with $\varepsilon \geq 0$ is defined as
        \begin{align}
            S_\varepsilon \mathscr{G}_X(t) = \bigcup_{[t]^\varepsilon} \mathscr{G}_X
        \end{align}
        for $t \in \mathbb{R}$.
    \end{enumerate}
\end{definition}

\begin{proposition}\label{Prop:smoothing-DG}
    We say that the graph $S_\varepsilon \mathscr{G}_X$ is in fact a DG. (See \cite[Proposition 6.7]{KimMemoli2021}).
\end{proposition}

\begin{definition}[Tripod between Two Sets]
    Consider the sets $X$ and $Y$. The \textbf{tripod} $R$ between $X$ and $Y$ can be described in the following diagram:
    \begin{align}
        R: X \xleftarrow[]{\varphi_X} Z \xrightarrow[]{\varphi_Y} Y.
    \end{align}
    We also require that both $\varphi_X$ and $\varphi_Y$ are surjections.
\end{definition}

\begin{definition}[$\varepsilon-$Tripod between DGs]
    We say two DGS $\mathscr{G}_X = (V_X(\cdot), E_X(\cdot))$, and $\mathscr{G}_Y = (V_Y(\cdot), E_Y(\cdot))$  are $\varepsilon-$interleaved if a tripod $R$ can be found for the sets $X$ and $Y$ where 
    \begin{align}
        \mathscr{G}_X \xrightarrow[]{R} S_\varepsilon \mathscr{G}_Y \quad \textnormal{and} \quad \mathscr{G}_Y \xrightarrow[]{R} S_\varepsilon \mathscr{G}_X.
    \end{align}
    This pair of mappings is denoted as a particular $\varepsilon-$tripod of the DGs $\mathscr{G}_X$ and $\mathscr{G}_Y$.
\end{definition}

\begin{definition}[Interleaving Distance Between Two DGs]
    If we allow $\mathscr{G}_X$ and $\mathscr{G}_Y$ to DGs, the infimum for all possible $\varepsilon \geq 0$ such that a $\varepsilon-$tripod among the two DGs $\mathscr{G}_X$ and $\mathscr{G}_Y$ exists is denoted as the interleaving distance $d_1^{\textnormal{dynG}}(\mathscr{G}_X, \mathscr{G}_Y)$. In the case where an $\varepsilon-$tripod does not exist between $\mathscr{G}_X$ and $\mathscr{G}_Y$ for all possible $\varepsilon \geq 0$, we say that $d_1^{\textnormal{dynG}}(\mathscr{G}_X, \mathscr{G}_Y) = + \infty$
\end{definition}

\begin{definition}[Pseudo-Metric]
    A pseudo-metric $d$ on a set $S$ fulfills the same properties as a metric except we allow $d(x,y) = 0$ for some $x,y \in S$ and $x\neq y$. 
\end{definition}

\begin{definition}[Extended Pseudo-Metric]
    An extended pseudo-metric $d$ on a set $S$ is a pseudo metric where we allow $d(x,y) = \infty$ for some $x,y \in S$.
\end{definition}

\begin{theorem}\label{Thm:pseudo-metric}
    We say that $d_1^{\textnormal{dynG}}$ is an extended pseudo-metric on the set of DGs $\mathcal{N}$. (See \cite[Theorem~6.10]{KimMemoli2021})
\end{theorem}

\begin{definition}[Sub-equivalence Relation]
    Consider the relation $\sim$ on the set $X' \subset X$ where $X$ is a non-empty set. We say that the relation $\sim$ is a sub equivalence relation on the set $X$. We say that the set $X'$ is the underlying set for $\sim$. Moreover,
    \begin{align}
        X' \coloneq \{x \in X : (x,x) \in \sim \}.
    \end{align}
\end{definition}

\begin{definition}[The Sub-equivalence Closure]
    Suppose we have a collection of sub-equivalence relations for an index set $I$. We can define this set as $\{\sim_i \subset X \times X : i \in I\}$ where $\sim_i$ is an equivalence relation of the set $X$. We say that the sub-equivalence closure $\sim_c$ for the set $\{ \sim_i \subset X \times X : i \in I \}$ is the smallest sub-equivalence relation containing all $\sim_i$ given any $i \in I$.
\end{definition}

\begin{definition}[Finest Common Coarsening]
    Consider the set $\mathcal{P}^{\textnormal{sub}}(X)$. We will denote $\{ P_i \}_{i \in I}$ as any subcollection of $\mathcal{P}^{\textnormal{sub}}(X)$ where for each $i \in I$, we let $\sim_i$ be a sub-equivalence relation on $X$ corresponding to the set $P_i$. We define $\bigvee_{i \in I} P_i$ as the sub-partition of $X$ that is the sub-equivalence closure for the collection $\{ \sim_i \subset X \times X : i \in I \}$. We say $\bigvee_{i \in I}P_i$ is the finest common coarsening for $\{P_i\}_{i \in I}$.
\end{definition}

\begin{definition}[Pullback of A Sub-partition and A Formigram]
    Once again, we will consider the collection of subcollections $\mathcal{P}^{\textnormal{sub}}(X)$. Let $P_X \in \mathcal{P}^{\textnormal{sub}}(X)$ with the underlying set $X$. If we have any $\varphi : Z \to X$, we say that $P_Z \coloneq \varphi^*P_X$ is the \textbf{pullback} of $P_X$ via $\varphi$. Moreover, we say $P_Z$ is the sub-partition of $Z$ defined as $P_Z = \{ \varphi^{-1}(B) : B \in P_X \}$. Furthermore, if we have a set $X$ with the formigram $\theta_X$ over that set, the pullback of $\theta_X$ via $\varphi$ is defined as follows
    \begin{align}
        \theta_Z \coloneq \varphi^*\theta_X \textnormal{ over } Z \textnormal{ defined as } \theta_Z(t) = \varphi^* \theta_X(t), \forall t \in \mathbb{R}
    \end{align}
\end{definition}

\begin{notation}[The Underlying Set of $\theta_X(t)$]
    We will denote the \textbf{underlying set of} $\theta_X(t)$ for a formigram $\theta_X$ over the set $X$, and for all $t \in \mathbb{R}$ by $U_{\theta_X}(t) \subset X$
\end{notation}

\begin{definition}[Tripod of A Formigram]
    Let $X$ and $Y$ share a tripod $R: X \xleftarrow[]{\varphi_X} Z \xrightarrow[]{\varphi_Y} Y$. Then there exists a tripod $R$ between $\theta_X$ and $\theta_Y$, $\theta_X \xrightarrow[]{R} \theta_Y$, if and only if particular conditions are met for any fixed $t \in \mathbb{R}$. For the sake of brevity, these conditions can be found in \cite{KimMemoli2021}.
\end{definition}

\begin{definition}[Time-Interlevel Smoothing of A Formigram]
    We will define the notion of \textbf{time-interlevel smoothing for formigrams} by first considering a formigram $\theta_X$ over the set $X$. Now we must consider the following procedures
    \begin{enumerate}[(i)]
        \item For an interval $I \subset \mathbb{R}$, the finest common coarsening of $\{ \theta_X(t) : t \in I \}$ is defined as $\bigvee_I \theta_X$
        \item Given an $\varepsilon \geq 0$, the $\varepsilon-$\textbf{smoothing} $S_\varepsilon\theta_X$ for the formigram $\theta_X$ is defined by the following
        \begin{align}
            (S_\varepsilon \theta_X)(t) = \bigvee_{[t]^\varepsilon}\theta_X \textnormal{ for } t\in \mathbb{R}
        \end{align}
    \end{enumerate}
\end{definition}

\begin{proposition} The $\varepsilon-$smoothing
    $S_\varepsilon \theta_X$ is a formigram. (See \cite[Proposition 6.24]{KimMemoli2021}).
\end{proposition}

\begin{definition}[$\varepsilon-$Interleaved Formigrams]
    We say that two formigrams $\theta_X$ and $\theta_Y$ over the sets $X$ and $Y$ are $\varepsilon-$interleaved if $\exists R$ where $R$ is a tripod between $X$ and $Y$ such that
    \begin{align}
        \theta_X \xrightarrow[]{R} S_\varepsilon \theta_Y \quad \textnormal{and} \quad \theta_Y \xrightarrow[]{R}S_\varepsilon\theta_X
    \end{align}
    This $R$ is defined to be the $\varepsilon-$tripod between $\theta_X$ and $\theta_Y$.
\end{definition}

\begin{definition}[The Interleaving Distance between Formigrams]
    If we have two formigrams $\theta_X$ and $\theta_Y$ over the sets $X$ and $Y$, the interleaving distance $d_1^F(\theta_X, \theta_Y)$ between $\theta_X$ and $\theta_Y$ is defined as the value of $\varepsilon \geq 0$ where $\varepsilon$ is the infimum of all possible $\varepsilon$ where a $\varepsilon-$tripod between $\theta_X$ and $\theta_Y$ exists. Moreover, if this $\varepsilon-$tripod between $\theta_X$ and $\theta_Y$ doesn't exist for any $\varepsilon\geq 0$, then we say $d_1^F (\theta_X, \theta_Y) = + \infty$ 
\end{definition}

\begin{theorem}
    We say $d_1^F$ is an extended pseudo metric on the set of formigrams.
\end{theorem}

\begin{result}[The Stability Result]
    The path component functor $\pi_0$ allows DGs to be mapped into formigrams in a stable manner via the following theorems
\end{result}

\begin{theorem}[$\pi_0$ is $1-$Lipschitz]
    Consider the DGs $\mathscr{G}_X$ and $\mathscr{G}_Y$. If we let $\pi_0(\mathscr{G}_X)$ and $\pi_0(\mathscr{G}_Y)$ be formigrams defined by the path component functor, then we see that
    \begin{align}
        d_1^F(\theta_X, \theta_Y) \leq d_1^{\textnormal{dynG}}(\mathscr{G}_X, \mathscr{G}_Y)
    \end{align}
  (See \cite[Theorem 6.32]{KimMemoli2021}).
\end{theorem}

\begin{corollary}
    Consider the formigram $\theta_X$ over $X$. Then for an arbitrary $\varepsilon \geq 0$
    \begin{align}
        d_B(\textnormal{dgm}(S_\varepsilon \theta_X), \textnormal{dgm}(\theta_X)) \leq \varepsilon.
    \end{align}
    (See \cite[Corollary~7.2]{KimMemoli2021}).
\end{corollary}

\section{Dynamic Bayesian Networks}
The following section summarizes standard definitions from~\cite{Koller-Friedman}.
In particular, we will establish the theoretical framework necessary to understand Dynamic Bayesian Networks. 

Assume we have a set of random variables ${X_1, \dots ,X_n}$. In addition to this, let $P(X_1, \dots , X_n)$ represent a joint conditional probability distribution (CPD) of our random variables.

\subsection{Bayesian Networks}

\begin{definition} \cite[Definition 3.1]{Koller-Friedman}
    A Bayesian Network Structure $G$ is a directed acyclic graph (DAG) where:
    \begin{enumerate}[(i)]
        \item \textbf{Nodes:} represent random variables $X_1, \dots , X_n$.
        \item \textbf{Directed edges:} represent dependencies between random variables.
        \item \textbf{Graph:} encodes a set of local independencies among our random variables.
    \end{enumerate}
\end{definition}

\begin{example}
    Suppose we have a set of random variables $\mathcal{X} = \{X_1, \dots, X_5\}$, and suppose $(X_4 \not\perp X_1)$, $(X_5 \not\perp X_1), (X_4 \not\perp X_2)$, $(X_5 \not\perp X_2)$, $(X_5 \not\perp X_3)$ is an exhaustive list of all dependencies in our distribution. We may create the associated Bayesian Network Structure in Figure \ref{Fig:Bayesian-CPD}.
\end{example}

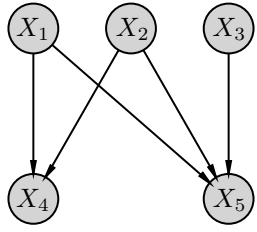
\begin{figure}[H]
    \begin{center}
        \tikzset{every picture/.style={line width=0.75pt}} 
        
        \begin{tikzpicture}[x=0.75pt,y=0.75pt,yscale=-0.5,xscale=0.5]
        
        \draw    (320.5,68) -- (418.5,239) ;
        \draw    (320.5,68) -- (222.5,239) ;
        \draw    (222.5,93) -- (222.5,214) ;
        \draw    (418.5,93) -- (418.5,214) ;
        \draw    (222.5,68) -- (317.28,150.69) -- (418.5,239) ;
        \draw  [fill={rgb, 255:red, 212; green, 212; blue, 212 }  ,fill opacity=1 ] (197.5,68) .. controls (197.5,54.19) and (208.69,43) .. (222.5,43) .. controls (236.31,43) and (247.5,54.19) .. (247.5,68) .. controls (247.5,81.81) and (236.31,93) .. (222.5,93) .. controls (208.69,93) and (197.5,81.81) .. (197.5,68) -- cycle ;
        \draw  [fill={rgb, 255:red, 212; green, 212; blue, 212 }  ,fill opacity=1 ] (197.5,239) .. controls (197.5,225.19) and (208.69,214) .. (222.5,214) .. controls (236.31,214) and (247.5,225.19) .. (247.5,239) .. controls (247.5,252.81) and (236.31,264) .. (222.5,264) .. controls (208.69,264) and (197.5,252.81) .. (197.5,239) -- cycle ;
        \draw  [fill={rgb, 255:red, 212; green, 212; blue, 212 }  ,fill opacity=1 ] (393.5,239) .. controls (393.5,225.19) and (404.69,214) .. (418.5,214) .. controls (432.31,214) and (443.5,225.19) .. (443.5,239) .. controls (443.5,252.81) and (432.31,264) .. (418.5,264) .. controls (404.69,264) and (393.5,252.81) .. (393.5,239) -- cycle ;
        \draw  [fill={rgb, 255:red, 212; green, 212; blue, 212 }  ,fill opacity=1 ] (393.5,68) .. controls (393.5,54.19) and (404.69,43) .. (418.5,43) .. controls (432.31,43) and (443.5,54.19) .. (443.5,68) .. controls (443.5,81.81) and (432.31,93) .. (418.5,93) .. controls (404.69,93) and (393.5,81.81) .. (393.5,68) -- cycle ;
        \draw  [fill={rgb, 255:red, 0; green, 0; blue, 0 }  ,fill opacity=1 ] (222.5,214) -- (220.2,201) -- (224.91,201.02) -- cycle ;
        \draw  [fill={rgb, 255:red, 0; green, 0; blue, 0 }  ,fill opacity=1 ] (418.5,214) -- (416.2,201) -- (420.91,201.02) -- cycle ;
        \draw  [fill={rgb, 255:red, 0; green, 0; blue, 0 }  ,fill opacity=1 ] (400.13,223.59) -- (389.26,216.1) -- (392.58,212.76) -- cycle ;
        \draw  [fill={rgb, 255:red, 212; green, 212; blue, 212 }  ,fill opacity=1 ] (295.5,68) .. controls (295.5,54.19) and (306.69,43) .. (320.5,43) .. controls (334.31,43) and (345.5,54.19) .. (345.5,68) .. controls (345.5,81.81) and (334.31,93) .. (320.5,93) .. controls (306.69,93) and (295.5,81.81) .. (295.5,68) -- cycle ;
        \draw  [fill={rgb, 255:red, 0; green, 0; blue, 0 }  ,fill opacity=1 ] (234.39,217.38) -- (239.72,205.31) -- (243.62,207.95) -- cycle ;
        \draw  [fill={rgb, 255:red, 0; green, 0; blue, 0 }  ,fill opacity=1 ] (405.93,217.82) -- (398.02,207.25) -- (402.23,205.14) -- cycle ;
        
        \draw (302.25,54.5) node [anchor=north west][inner sep=0.75pt]   [align=left] {$\displaystyle X_{2}$};
        \draw (215.25,58) node [anchor=north west][inner sep=0.75pt]   [align=left] {$ $};
        \draw (399.25,54.5) node [anchor=north west][inner sep=0.75pt]   [align=left] {$\displaystyle X_{3}$};
        \draw (202.25,54.5) node [anchor=north west][inner sep=0.75pt]   [align=left] {$\displaystyle X_{1}$};
        \draw (202.25,225.5) node [anchor=north west][inner sep=0.75pt]   [align=left] {$\displaystyle X_{4}$};
        \draw (399.75,225.5) node [anchor=north west][inner sep=0.75pt]   [align=left] {$\displaystyle X_{5}$};
        \end{tikzpicture} 
        \caption{Bayesian Network Representing our CPD} \label{Fig:Bayesian-CPD}
    \end{center}
\end{figure}

\begin{definition}
    Let $Pa^{G}_{X_i}$ denote the parents of $X_i$ in $G$ and $NonDescendants_{X_i}$ denote the variables in the graph that are not descendants of $X_i$.
    \begin{enumerate}[(i)]
        \item $G$ encodes the following set of conditional independence assumptions, called the local independencies, denoted $\mathcal{I}_{\ell}(G):$
        \item For each variable $X_i$: $(X_i \perp NonDescendants_{X_i} | Pa^{G}_{X_i}).$
    \end{enumerate}
\end{definition}
The second condition is the local Markov condition: each node is conditionally independent of all its non-descendants given its parents. This is the defining property that lets the joint distribution factor as $P(\mathcal{X}) = \prod_i P(X_i \mid \mathrm{Pa}^G_{X_i})$. (See \cite[Theorem~3.5]{Koller-Friedman}).

\begin{definition}
    Let \( G \) be a Bayesian Network (BN) Structure over the variables \( X_1, \ldots, X_n \). We say that a distribution \( P \) over the same space factorizes according to \( G \) if \( P \) can be expressed as a product:
    \[
    P(X_1, \ldots, X_n) = P(\mathcal{X}) = \prod_{i=1}^{n} P(X_i \mid \text{Pa}^G_{X_i}).
    \]

    This equation is called the chain rule for Bayesian networks.
\end{definition}

\begin{definition}
    A Bayesian network is a pair \( \mathcal{B} = (G, P) \), where \( P \) factorizes over \( G \), and where \( P \) is specified as a set of conditional probability distributions (CPDs) associated with the nodes of \( G \). The distribution \( P \) is often annotated as \( P_{\mathcal{B}} \).
\end{definition}

\begin{remark}
    Consider the following:
    \begin{enumerate}[(i)]
     \item \textit{Bayesian Networks are an efficient way of representing a joint distribution over random variables $\mathcal{X} = \{X_1, \dots, X_n \}$}
     \item \textit{Bayesian Networks allow for a graphical representation of a CPD.}
     \item \textit{It is important to note that the independencies present in a BN may not capture all independencies in a CPD}
    \end{enumerate}
\end{remark}

\subsection{Dynamic Bayesian Networks}

\begin{convention}
    We can now consider a temporal model representing our network
    \begin{enumerate}[(i)]
         \item \textit{We can now create a time increment \(\Delta t \) and a set of random variable sets \( \mathcal{X}^{(0)}, \mathcal{X}^{(1)}, \ldots \) where \( \mathcal{X}^{(k)} \) are the random variables that represent the system state at time \( k \cdot \Delta t \).}
         \item \textit{We use \( X_i^{(k)} \) to represent the instantiation of a random variable \( X_i \) at time \( k \) where \( X_i^{(k)} \in \mathcal{X}^{(k)} \).}
         \item \textit{\( P(\mathcal{X}^{(0)}, \mathcal{X}^{(1)}, \ldots, \mathcal{X}^{(K)}) \) with \( k = 0, \ldots, K \), often abbreviated as \( P(\mathcal{X}^{(0:K)})\), will represent the future trajectories of our system.}
    \end{enumerate}
\end{convention}

\begin{definition} \cite[Definition 6.1]{Koller-Friedman}
    A process over $\mathcal{X}$ means there is a sequence of random variable sets $\mathcal{X}^{(0)}, \mathcal{X}^{(1)}, \ldots ,$ where $\mathcal{X}^{(k)}$ is a copy of the variables in $\mathcal{X}$ at time $k$.
\end{definition}

\begin{definition}\cite[Definition 6.2]{Koller-Friedman}
    A 2-time-slice Bayesian network (2-TBN) for a process over \( \mathcal{X} \) is a conditional Bayesian network over \( \mathcal{X}' \) given \( \mathcal{X}_I \), where \( \mathcal{X}_I \subseteq \mathcal{X} \) is a set of interface variables (in red in Figure \ref{Fig: time-slices}).
\end{definition}

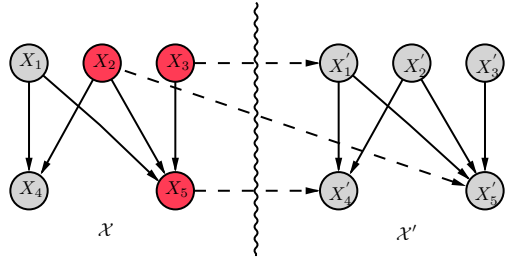
\begin{figure}
    \begin{centering}

        \tikzset{every picture/.style={line width=0.75pt}} 
        
        \begin{tikzpicture}[x=0.75pt,y=0.75pt,yscale=-0.6,xscale=0.6]
        
        \draw  [dash pattern={on 4.5pt off 4.5pt}]  (196.31,88.61) -- (498.86,190.86) ;
        \draw [shift={(500.75,191.5)}, rotate = 198.67] [fill={rgb, 255:red, 0; green, 0; blue, 0 }  ][line width=0.08]  [draw opacity=0] (12,-3) -- (0,0) -- (12,3) -- cycle    ;
        \draw    (196.31,88.61) -- (248.26,180.26) ;
        \draw [shift={(249.25,182)}, rotate = 240.45] [fill={rgb, 255:red, 0; green, 0; blue, 0 }  ][line width=0.08]  [draw opacity=0] (12,-3) -- (0,0) -- (12,3) -- cycle    ;
        \draw    (196.31,88.61) -- (145.72,180.25) ;
        \draw [shift={(144.75,182)}, rotate = 298.9] [fill={rgb, 255:red, 0; green, 0; blue, 0 }  ][line width=0.08]  [draw opacity=0] (12,-3) -- (0,0) -- (12,3) -- cycle    ;
        \draw    (135.11,104.22) -- (135.11,177.78) ;
        \draw [shift={(135.11,179.78)}, rotate = 270] [fill={rgb, 255:red, 0; green, 0; blue, 0 }  ][line width=0.08]  [draw opacity=0] (12,-3) -- (0,0) -- (12,3) -- cycle    ;
        \draw    (257.5,104.22) -- (257.5,177.78) ;
        \draw [shift={(257.5,179.78)}, rotate = 270] [fill={rgb, 255:red, 0; green, 0; blue, 0 }  ][line width=0.08]  [draw opacity=0] (12,-3) -- (0,0) -- (12,3) -- cycle    ;
        \draw    (135.11,88.61) -- (194.29,140.24) -- (242.8,186.62) ;
        \draw [shift={(244.25,188)}, rotate = 223.71] [fill={rgb, 255:red, 0; green, 0; blue, 0 }  ][line width=0.08]  [draw opacity=0] (12,-3) -- (0,0) -- (12,3) -- cycle    ;
        \draw  [fill={rgb, 255:red, 212; green, 212; blue, 212 }  ,fill opacity=1 ] (119.5,88.61) .. controls (119.5,79.99) and (126.49,73) .. (135.11,73) .. controls (143.73,73) and (150.72,79.99) .. (150.72,88.61) .. controls (150.72,97.23) and (143.73,104.22) .. (135.11,104.22) .. controls (126.49,104.22) and (119.5,97.23) .. (119.5,88.61) -- cycle ;
        \draw  [fill={rgb, 255:red, 212; green, 212; blue, 212 }  ,fill opacity=1 ] (119.5,195.39) .. controls (119.5,186.77) and (126.49,179.78) .. (135.11,179.78) .. controls (143.73,179.78) and (150.72,186.77) .. (150.72,195.39) .. controls (150.72,204.01) and (143.73,211) .. (135.11,211) .. controls (126.49,211) and (119.5,204.01) .. (119.5,195.39) -- cycle ;
        \draw  [fill={rgb, 255:red, 255; green, 59; blue, 85 }  ,fill opacity=1 ] (241.89,195.39) .. controls (241.89,186.77) and (248.88,179.78) .. (257.5,179.78) .. controls (266.12,179.78) and (273.11,186.77) .. (273.11,195.39) .. controls (273.11,204.01) and (266.12,211) .. (257.5,211) .. controls (248.88,211) and (241.89,204.01) .. (241.89,195.39) -- cycle ;
        \draw  [fill={rgb, 255:red, 255; green, 59; blue, 85 }  ,fill opacity=1 ] (241.89,88.61) .. controls (241.89,79.99) and (248.88,73) .. (257.5,73) .. controls (266.12,73) and (273.11,79.99) .. (273.11,88.61) .. controls (273.11,97.23) and (266.12,104.22) .. (257.5,104.22) .. controls (248.88,104.22) and (241.89,97.23) .. (241.89,88.61) -- cycle ;
        \draw  [fill={rgb, 255:red, 255; green, 59; blue, 85 }  ,fill opacity=1 ] (180.69,88.61) .. controls (180.69,79.99) and (187.68,73) .. (196.31,73) .. controls (204.93,73) and (211.92,79.99) .. (211.92,88.61) .. controls (211.92,97.23) and (204.93,104.22) .. (196.31,104.22) .. controls (187.68,104.22) and (180.69,97.23) .. (180.69,88.61) -- cycle ;
        \draw    (455.31,88.61) -- (508.74,180.27) ;
        \draw [shift={(509.75,182)}, rotate = 239.76] [fill={rgb, 255:red, 0; green, 0; blue, 0 }  ][line width=0.08]  [draw opacity=0] (12,-3) -- (0,0) -- (12,3) -- cycle    ;
        \draw    (455.31,88.61) -- (404.22,180.25) ;
        \draw [shift={(403.25,182)}, rotate = 299.14] [fill={rgb, 255:red, 0; green, 0; blue, 0 }  ][line width=0.08]  [draw opacity=0] (12,-3) -- (0,0) -- (12,3) -- cycle    ;
        \draw    (394.11,104.22) -- (394.11,177.78) ;
        \draw [shift={(394.11,179.78)}, rotate = 270] [fill={rgb, 255:red, 0; green, 0; blue, 0 }  ][line width=0.08]  [draw opacity=0] (12,-3) -- (0,0) -- (12,3) -- cycle    ;
        \draw    (516.5,104.22) -- (516.5,177.78) ;
        \draw [shift={(516.5,179.78)}, rotate = 270] [fill={rgb, 255:red, 0; green, 0; blue, 0 }  ][line width=0.08]  [draw opacity=0] (12,-3) -- (0,0) -- (12,3) -- cycle    ;
        \draw    (394.11,88.61) -- (453.29,140.24) -- (502.77,185.16) ;
        \draw [shift={(504.25,186.5)}, rotate = 222.23] [fill={rgb, 255:red, 0; green, 0; blue, 0 }  ][line width=0.08]  [draw opacity=0] (12,-3) -- (0,0) -- (12,3) -- cycle    ;
        \draw  [fill={rgb, 255:red, 212; green, 212; blue, 212 }  ,fill opacity=1 ] (378.5,88.61) .. controls (378.5,79.99) and (385.49,73) .. (394.11,73) .. controls (402.73,73) and (409.72,79.99) .. (409.72,88.61) .. controls (409.72,97.23) and (402.73,104.22) .. (394.11,104.22) .. controls (385.49,104.22) and (378.5,97.23) .. (378.5,88.61) -- cycle ;
        \draw  [fill={rgb, 255:red, 212; green, 212; blue, 212 }  ,fill opacity=1 ] (378.5,195.39) .. controls (378.5,186.77) and (385.49,179.78) .. (394.11,179.78) .. controls (402.73,179.78) and (409.72,186.77) .. (409.72,195.39) .. controls (409.72,204.01) and (402.73,211) .. (394.11,211) .. controls (385.49,211) and (378.5,204.01) .. (378.5,195.39) -- cycle ;
        \draw  [fill={rgb, 255:red, 212; green, 212; blue, 212 }  ,fill opacity=1 ] (500.89,195.39) .. controls (500.89,186.77) and (507.88,179.78) .. (516.5,179.78) .. controls (525.12,179.78) and (532.11,186.77) .. (532.11,195.39) .. controls (532.11,204.01) and (525.12,211) .. (516.5,211) .. controls (507.88,211) and (500.89,204.01) .. (500.89,195.39) -- cycle ;
        \draw  [fill={rgb, 255:red, 212; green, 212; blue, 212 }  ,fill opacity=1 ] (500.89,88.61) .. controls (500.89,79.99) and (507.88,73) .. (516.5,73) .. controls (525.12,73) and (532.11,79.99) .. (532.11,88.61) .. controls (532.11,97.23) and (525.12,104.22) .. (516.5,104.22) .. controls (507.88,104.22) and (500.89,97.23) .. (500.89,88.61) -- cycle ;
        \draw  [fill={rgb, 255:red, 212; green, 212; blue, 212 }  ,fill opacity=1 ] (439.69,88.61) .. controls (439.69,79.99) and (446.68,73) .. (455.31,73) .. controls (463.93,73) and (470.92,79.99) .. (470.92,88.61) .. controls (470.92,97.23) and (463.93,104.22) .. (455.31,104.22) .. controls (446.68,104.22) and (439.69,97.23) .. (439.69,88.61) -- cycle ;
        \draw    (325,38) .. controls (326.67,39.67) and (326.67,41.33) .. (325,43) .. controls (323.33,44.67) and (323.33,46.33) .. (325,48) .. controls (326.67,49.67) and (326.67,51.33) .. (325,53) .. controls (323.33,54.67) and (323.33,56.33) .. (325,58) .. controls (326.67,59.67) and (326.67,61.33) .. (325,63) .. controls (323.33,64.67) and (323.33,66.33) .. (325,68) .. controls (326.67,69.67) and (326.67,71.33) .. (325,73) .. controls (323.33,74.67) and (323.33,76.33) .. (325,78) .. controls (326.67,79.67) and (326.67,81.33) .. (325,83) .. controls (323.33,84.67) and (323.33,86.33) .. (325,88) .. controls (326.67,89.67) and (326.67,91.33) .. (325,93) .. controls (323.33,94.67) and (323.33,96.33) .. (325,98) .. controls (326.67,99.67) and (326.67,101.33) .. (325,103) .. controls (323.33,104.67) and (323.33,106.33) .. (325,108) .. controls (326.67,109.67) and (326.67,111.33) .. (325,113) .. controls (323.33,114.67) and (323.33,116.33) .. (325,118) .. controls (326.67,119.67) and (326.67,121.33) .. (325,123) .. controls (323.33,124.67) and (323.33,126.33) .. (325,128) .. controls (326.67,129.67) and (326.67,131.33) .. (325,133) .. controls (323.33,134.67) and (323.33,136.33) .. (325,138) .. controls (326.67,139.67) and (326.67,141.33) .. (325,143) .. controls (323.33,144.67) and (323.33,146.33) .. (325,148) .. controls (326.67,149.67) and (326.67,151.33) .. (325,153) .. controls (323.33,154.67) and (323.33,156.33) .. (325,158) .. controls (326.67,159.67) and (326.67,161.33) .. (325,163) .. controls (323.33,164.67) and (323.33,166.33) .. (325,168) .. controls (326.67,169.67) and (326.67,171.33) .. (325,173) .. controls (323.33,174.67) and (323.33,176.33) .. (325,178) .. controls (326.67,179.67) and (326.67,181.33) .. (325,183) .. controls (323.33,184.67) and (323.33,186.33) .. (325,188) .. controls (326.67,189.67) and (326.67,191.33) .. (325,193) .. controls (323.33,194.67) and (323.33,196.33) .. (325,198) .. controls (326.67,199.67) and (326.67,201.33) .. (325,203) .. controls (323.33,204.67) and (323.33,206.33) .. (325,208) .. controls (326.67,209.67) and (326.67,211.33) .. (325,213) .. controls (323.33,214.67) and (323.33,216.33) .. (325,218) .. controls (326.67,219.67) and (326.67,221.33) .. (325,223) .. controls (323.33,224.67) and (323.33,226.33) .. (325,228) .. controls (326.67,229.67) and (326.67,231.33) .. (325,233) .. controls (323.33,234.67) and (323.33,236.33) .. (325,238) .. controls (326.67,239.67) and (326.67,241.33) .. (325,243) .. controls (323.33,244.67) and (323.33,246.33) .. (325,248) -- (325,250) -- (325,250) ;
        \draw  [dash pattern={on 4.5pt off 4.5pt}]  (273.11,88.61) -- (376.5,88.61) ;
        \draw [shift={(378.5,88.61)}, rotate = 180] [fill={rgb, 255:red, 0; green, 0; blue, 0 }  ][line width=0.08]  [draw opacity=0] (12,-3) -- (0,0) -- (12,3) -- cycle    ;
        \draw  [dash pattern={on 4.5pt off 4.5pt}]  (273.11,195.39) -- (376.5,195.39) ;
        \draw [shift={(378.5,195.39)}, rotate = 180] [fill={rgb, 255:red, 0; green, 0; blue, 0 }  ][line width=0.08]  [draw opacity=0] (12,-3) -- (0,0) -- (12,3) -- cycle    ;
        
        \draw (186.65,79.11) node [anchor=north west][inner sep=0.75pt]  [xscale=0.7,yscale=0.7] [align=left] {$\displaystyle X_{2}$};
        \draw (128.52,78.8) node [anchor=north west][inner sep=0.75pt]  [xscale=0.7,yscale=0.7] [align=left] {$ $};
        \draw (248.22,79.42) node [anchor=north west][inner sep=0.75pt]  [xscale=0.7,yscale=0.7] [align=left] {$\displaystyle X_{3}$};
        \draw (125.2,78.49) node [anchor=north west][inner sep=0.75pt]  [xscale=0.7,yscale=0.7] [align=left] {$\displaystyle X_{1}$};
        \draw (125.2,185.89) node [anchor=north west][inner sep=0.75pt]  [xscale=0.7,yscale=0.7] [align=left] {$\displaystyle X_{4}$};
        \draw (247.28,185.89) node [anchor=north west][inner sep=0.75pt]  [xscale=0.7,yscale=0.7] [align=left] {$\displaystyle X_{5}$};
        \draw (446.65,79.11) node [anchor=north west][inner sep=0.75pt]  [xscale=0.7,yscale=0.7] [align=left] {$\displaystyle X_{2}^{'}$};
        \draw (387.52,78.8) node [anchor=north west][inner sep=0.75pt]  [xscale=0.7,yscale=0.7] [align=left] {$ $};
        \draw (507.22,79.42) node [anchor=north west][inner sep=0.75pt]  [xscale=0.7,yscale=0.7] [align=left] {$\displaystyle X_{3}^{'}$};
        \draw (384.2,78.49) node [anchor=north west][inner sep=0.75pt]  [xscale=0.7,yscale=0.7] [align=left] {$\displaystyle X_{1}^{'}$};
        \draw (384.2,185.89) node [anchor=north west][inner sep=0.75pt]  [xscale=0.7,yscale=0.7] [align=left] {$\displaystyle X_{4}^{'}$};
        \draw (506.28,185.89) node [anchor=north west][inner sep=0.75pt]  [xscale=0.7,yscale=0.7] [align=left] {$\displaystyle X_{5}^{'}$};
        \draw (190.8,220.6) node [anchor=north west][inner sep=0.75pt]  [xscale=0.7,yscale=0.7] [align=left] {$\displaystyle \mathcal{X}$};
        \draw (440.4,222.2) node [anchor=north west][inner sep=0.75pt]  [xscale=0.7,yscale=0.7] [align=left] {$\displaystyle \mathcal{X} '$};
    
        \end{tikzpicture}
    \caption{Two time slices of a DBN with interface variables $X_2$, $X_3$ and $X_5$ (in red).} 
    \label{Fig: time-slices}

    \end{centering}
\end{figure}

\begin{definition}
    A dynamic Bayesian network (DBN) is a pair \( \langle \mathcal{B}_0, \mathcal{B}_{\rightarrow} \rangle \), where \( \mathcal{B}_0 \) is a Bayesian network over \( \mathcal{X}^{(0)} \), representing the initial distribution over states, and \( \mathcal{B}_{\rightarrow} \) is a 2-TBN for the process. For any desired time span \( K \geq 0 \), the distribution over \( \mathcal{X}^{(0:K)} \) is defined as an unrolled Bayesian network, where, for any \( i = 1, \ldots, n \):

    \begin{enumerate}[(i)]
     \item \textit{The structure and CPDs of \( X_i^{(0)} \) are the same as those for \( X_i \) in \( \mathcal{B}_0 \).}
     \item \textit{The structure and CPD of \( X_i^{(k)} \) for \( k > 0 \) are the same as those for \( X_i' \) in \( \mathcal{B}_{\rightarrow} \).}
    \end{enumerate}
\end{definition}

\section{Determining Edge Weights}

In this section, we define an edge weight function between nodes
in a Bayesian Network following \cite{diameter-LeonelliSmith2025}. This definition is associated with the diameter of a stochastic matrix and, more generally, can be described as the maximum diameter of a stochastic sub-Conditional Probability Table representing the probability distribution of a node conditioned on some specification of its parents.

\begin{definition}\cite[Definition 1]{diameter-LeonelliSmith2025}
    Let $p$ and $p'$ be two probability mass functions over the same sample space $\mathcal{X}$. The \textit{total variation distance} between $p$ and $p'$ is
    \[
    d_V(p,p') = \frac{1}{2} \sum\limits_{x \in \mathcal{X}} |p(x)-p'(x)|
    \]
\end{definition}

\begin{remark}
    The total variation distance is a way of gauging how different two probability distributions are.
\end{remark}

\begin{definition}\cite[Definition 2]{diameter-LeonelliSmith2025}
    The \textit{upper diameter} of a $n \times k$ stochastic matrix  $P$ with rows $p_1,\cdots,p_n$ denoted as $d^+(P)$, is
    \[
    d^+(P) =\max\limits_{\substack{i,j\in[n] \\\\ i \neq j}}d_V(p_i,p_j)
    \]

    The \textit{lower diameter} is
    \[
    d^-(P) = \min\limits_{\substack{i,j\in[n] \\\\ i \neq j}} d_V(p_i, p_j)
    \]

\end{definition}

\begin{remark}
    These equations determine either the maximum or the minimum total variation distance between the rows of a Conditional Probability Table (CPT).
\end{remark}

\begin{proposition}\cite[Proposition 1]{diameter-LeonelliSmith2025} \label{Prop:mutual-info}
    For two categorical random variables $X$ and $Y$, let $MI(X,Y)$ be their mutual information. It holds that   
    \[
    d^+(P_{Y|X})=0 \iff Y \perp X \iff MI(X,Y)=0.
    \]
\end{proposition}

\begin{remark}
    We can see intuitively from this proposition that the closer the diameter is to zero the more independent our two variables are. 
\end{remark}

In combination with the total variation distance and the upper diameter of a stochastic matrix, we can define a new measure of edge strength between nodes in a Bayesian Network (BN).

Let $P_i$ be the CPT of some node $X_i$ and $P_{i|\textbf{x}}$ the sub-CPT of $P_i$ including only the rows specified by $\textbf{x}$

\begin{definition}\label{Def:edge-strength}\cite[Definition 3]{diameter-LeonelliSmith2025}
    The edge strength of edge $(j,i)$ in a BN is defined as
    \[\delta_{ji} = \max_{\mathbf{x} \in \mathcal{X}{\Pi_i \backslash j}} d^+(P_{i|\mathbf{x}}).\]
\end{definition}

\begin{remark}\label{Rmk:indepence-interpretation}
The edge strength $\delta_{ji}$ has a precise conditional independence
interpretation~\cite[Proposition~4]{diameter-LeonelliSmith2025}: for $X_j \in \mathrm{Pa}^G_{X_i}$,
\[
  \delta_{ji} = 0
  \;\Longleftrightarrow\;
  X_i \perp X_j \mid X_{\Pi_i \setminus j},
\]

where $\Pi_i \setminus j$ denotes the remaining parents of $X_i$ in $G$.
That is, $\delta_{ji} = 0$ if and only if $X_j$ contributes no additional
information about $X_i$ once all other parents are known.

This gives the threshold $\eta$ in Definition~\ref{Def:dynamic-bayesian} a natural meaning:
an edge $\{X_j, X_i\}$ is included in $E_\eta(t)$ precisely when $X_j$
has a \emph{non-negligible} conditional influence on $X_i$ beyond what
the remaining parents already explain.  Edges with $\delta_{ji} \leq \eta$
are excluded because they represent parent--child pairs that are
approximately conditionally independent given the rest.
\end{remark}

Note that this is strictly stronger than Proposition \ref{Prop:mutual-info}, which characterizes $d^+(P_{Y|X}) = 0$, that is, conditional independence of $X_i$ from $X_j$ given a fixed configuration of the remaining parents $X_{\Pi_i \setminus j}$. Remark \ref{Rmk:indepence-interpretation} asserts that $\delta_{ji} = 0$ if and only if this conditional independence holds uniformly across all parent configurations.

\section{Stable Distance Persistent Homology for Dynamic Bayesian Networks}\label{Sec:stable-distance}

\subsection{Dynamic Bayesian Graphs}

\subsubsection{Modifications to Edge Sets}

Before we can use the methods of Stable Distance Persistent homology outlined in \cite{KimMemoli2021}, we must first fit the notion of a DBN to a Dynamic Graph. Since a DBN is considered to be a process that copies a BN over a discrete number of intervals in a given timespan, this generally corresponds to the edge structure of the network remaining constant over time. This corresponds to a scenario where edges are never born and never die. If we were to consider performing SDPH on this kind of graph, this would correspond to what is called a saturated and constant DG representation of a DBN. In other words the representation would have $\mathscr{G}_\mathcal{X} = (V_\mathcal{X}(\cdot), E_\mathcal{X}(\cdot))$ where $V_\mathcal{X}\mathcal{}(\cdot) \equiv \mathcal{X}$ and  $E_\mathcal{X}(\cdot)$ is constant. In this model, clustering barcodes would end up being trivial and no evolving dynamic signatures would be obtained. 

In order to combat this issue, we can reorient the notion of an edge with respect to a DBN. As established in previous sections, we can define an edge strength $w_{(\cdot,\cdot)}$ between nodes in a DBN. When creating a DBN, the edge set is defined as a pre-established and fixed set of directed edges $\overrightarrow{E}$. This set includes instantiations of the edge set in a BN at each time interval $\Delta t$ along with the inter-time-slice edges representing conditional dependence across times. This edge set can be made more sensible if we allow for a time index that \textit{excludes} the inter-time-slice edges graphically.

So let $\overrightarrow{E}(\cdot): \mathbb{R}_+ \to  \overrightarrow{E}$ be a function of time to our edge set such that if $k_1\Delta t < t' \leq k_2\Delta t$ $t' \mapsto \overrightarrow{E}(k_1)$ where 
\begin{align}
\overrightarrow{E}(t') \coloneq \left \{ (Y^{(k_2)}, X^{(k_2)}  )| Y^{(k_2)} \in \textnormal{Pa}_{X^{(k_2)}}^G \textnormal{ for } (Y^{(k_2)}, X^{(k_2)}  )\in \mathcal{X}^{(k_2)} \times \mathcal{X}^{(k_2)} \right \}.
\end{align}

Furthermore, we can alter this definition to only include edges between children and parents that have an edge weight above some tolerance $\eta$. In essence, this allows us to only have edges in our model whose parents have a significant enough impact on the state of their children,
\begin{align}
\overrightarrow{E}_\eta(t') \coloneq \left \{ (Y^{(k_2)}, X^{(k_2)}  ) \in \overrightarrow{E}(t') | w(Y^{(k_2)}, X^{(k_2)}) > \eta \right \},
\end{align}
where for each time $k_2$, we write $w(Y^{(k_2)}, X^{(k_2)}) \coloneqq \delta^{(k_2)}_{YX}$ for the edge strength between $Y$ and $X$ computed from the CPT of $\mathcal{B}_\rightarrow$ at time slice $k_2$.

Furthermore, since a DBN is inherently directed and acyclic clustering would be trivial if we were to only consider our model edges to be directed. So, in order to obtain a more rich clustering structure, we will allow our new edge set to be undirected. 
\begin{align}
    E_\eta(t') \coloneq \left \{ \overrightarrow{E}_\eta(t') \textnormal{ where edges are undirected} \right \}
\end{align}
In essence, each cluster now represents a group of random variables whose probabilistic behavior is tightly coupled: any two variables in the same cluster are linked by a path of edges each exceeding threshold $\eta$, making the cluster a natural unit for joint inference.

\subsubsection{Dynamic Bayesian Graphs}

Now that we have established an edge set that would induce a rich clustering structure, we will define a \textit{Dynamic Bayesian Graph} (DBG).

\begin{convention}\label{Conv:node-id}
[Node Identification] In a DBN $\langle \mathcal{B}_0, \mathcal{B}_\rightarrow \rangle$, the same set of random variables $\mathcal{X} = {X_1, \ldots, X_n}$ recurs at every time step. Although $X_i^{(t)}$ and $X_i^{(t')}$ may carry different CPTs when $t \neq t'$, they represent the same variable in the graph structure. We therefore identify $X_i^{(t)}$ with $X_i$ as a node, writing $\mathcal{X}_\rightarrow^{(t)} = \mathcal{X}$ for all $t \in \mathbb{R}+$. Consequently, the vertex set of the DBG is always the fixed set $\mathcal{X}$, and all time-dependence is encoded in the edge weight function $E_\eta(\cdot)$ rather than in the node set.
\end{convention}

\begin{definition}\label{Def:dynamic-bayesian}[Dynamic Bayesian Graph]
Consider a DBN $\langle \mathcal{B}_0, \mathcal{B}_{\to} \rangle$.
By Convention~\ref{Conv:node-id}, the vertex set is constant: $\mathcal{X}_{\to}^{(t)} \equiv \mathcal{X}$ for all $t \in \mathbb{R}_+$.
A \textbf{dynamic Bayesian graph} $\mathscr{G}_{\mathcal{X}} = (\mathcal{X}_{\to}^{(\cdot)}, E_\eta(\cdot))$ over $\mathcal{X}$ is the DG with constant vertex set $\mathcal{X}$, where the edge function $E_\eta(\cdot)$ satisfies the conditions below.

    \begin{enumerate}[(i)]
        \item (Self-loops) For all $t \in \mathbb{R}_+$ and for all $X\in \mathcal{X}_{\to}^{(t)}, (X, X) \in E_\eta(t)$.
Self-loops are included by definition independently of the threshold $\eta$; the edge-strength formula $\delta_{ji}$ is defined only for distinct parent-child pairs $j \neq i$ and does not apply to $(X,X)$.

        \item (Tameness) the set $\textnormal{crit}(\mathscr{G}_\mathcal{X})$ is locally finite.
        \item (Lifespan of vertices) for every $X \in \mathcal{X}$, the set $I_\mathcal{X} \coloneq \{ t \in \mathbb{R}_+ | X \in \mathcal{X_{\to}}^{(t)} \}$, said to be the \textbf{lifespan} of $X$, is a non-empty closed interval. 
        \item (Comparability) for every $t \in \mathbb{R}_+$, it holds that
        \begin{align}
            \mathcal{X_{\to}}^{(t - \varepsilon)} \subset \mathcal{X_{\to}}^{(t)} \supset \mathcal{X_{\to}}^{(t + \varepsilon)} \quad \textnormal{and} \quad E_\eta(t- \varepsilon) \subset E_\eta(t) \supset E_\eta(t+ \varepsilon)
        \end{align}
        for small $\varepsilon > 0$. We can rewrite this phrase as $\mathscr{G}_\mathcal{X}(t- \varepsilon) \subset \mathscr{G}_\mathcal{X}(t) \supset \mathscr{G}_\mathcal{X}(t + \varepsilon)$)
    \end{enumerate}
\end{definition}

\begin{definition}\label{Def:time-edge-strength}[Time-Indexed Edge Strength]  For a DBN $\langle \mathcal{B}_0, \mathcal{B}_{\to} \rangle$, let $\delta^{(k)}_{YX}$ denote the edge strength $\delta_{YX}$ computed from the CPT of $\mathcal{B}_\to$ at time slice $k$.
    
\end{definition}

\begin{proposition}\label{Prop:dynamic-Bayesian-dynamic-graph}
    A dynamic Bayesian graph is a dynamic graph.
\end{proposition}

\begin{proof}
We verify the four conditions of Definition \ref{Def:Dynamic-graph} for the DBG
$\mathscr{G}_{\mathcal{X}} = (\mathcal{X}_\rightarrow^{(\cdot)},
E_\eta(\cdot))$.

\medskip
\noindent\textbf{(i) Self-loops.}
By Definition \ref{Def:dynamic-bayesian}(i), $(X, X) \in E_\eta(t)$ for every
$X \in \mathcal{X}$ and every $t \in \BR_+$.  This is
the self-loop condition of Definition \ref{Def:Dynamic-graph}(i).

\medskip
\noindent\textbf{(ii) Tameness.}
Since $E_\eta(\cdot)$ is piecewise constant, a time $t$ can be critical
for $\mathscr{G}_{\mathcal{X}}$ only if $t \in \{k\Delta t : k \in \BN\}$.
Hence
\[
  \mathrm{crit}(\mathscr{G}_{\mathcal{X}})
  \;\subseteq\;
  \{k\Delta t : k \in \BN\},
\]
which is locally finite (at most one critical time per interval of length
$\Delta t$).

\medskip
\noindent\textbf{(iii) Lifespan of vertices.}
Since $\mathcal{X}_\rightarrow^{(t)} = \mathcal{X}$ for all $t \in \BR_+$,
the lifespan of every node $X \in \mathcal{X}$ is
$I_X = \BR_+$, which is a non-empty closed interval.

\medskip
\noindent\textbf{(iv) Comparability.}
We must show that for every $t \in \BR_+$ and all sufficiently small
$\varepsilon > 0$,
\[
  E_\eta(t - \varepsilon) \;\subset\; E_\eta(t) \;\supset\; E_\eta(t + \varepsilon).
\]
(Vertex comparability is trivial: all three vertex sets equal
$\mathcal{X}$.)

\smallskip
\noindent\textit{Between critical times.}
If $t$ is not a critical point of $\mathscr{G}_{\mathcal{X}}$, then
$t \in ((k-1)\Delta t, k\Delta t)$ for some $k$.  Taking
$\varepsilon < \min(t - (k-1)\Delta t,\; k\Delta t - t)$ gives
$E_\eta(t - \varepsilon) = E_\eta(t) = E_\eta(t + \varepsilon)$,
so the inclusions hold with equality.

\smallskip
\noindent\textit{At a critical time.}
Let $t^* = k\Delta t$ for some $k \in \BN$.  The left and right
limits of the edge set at $t^*$ are
\begin{align*}
  E_\eta^- &\;:=\; \lim_{\varepsilon \to 0} E_\eta(t^* - \varepsilon)
            \;=\; \bigl\{ \{Y,X\} : \delta^{(k-1)}_{YX} > \eta \bigr\}, \\[4pt]
  E_\eta^+ &\;:=\; \lim_{\varepsilon \to 0} E_\eta(t^* + \varepsilon)
            \;=\; \bigl\{ \{Y,X\} : \delta^{(k)}_{YX} > \eta \bigr\}.
\end{align*}
We adopt the \emph{maximum convention}: at each critical time $t^* = k\Delta t$,
set
\[
  \delta^{(k\Delta t)}_{YX} \;:=\; \max\!\bigl(\delta^{(k-1)}_{YX},\, \delta^{(k)}_{YX}\bigr),
  \qquad \text{so} \qquad
  E_\eta(k\Delta t) \;=\; E_\eta^- \cup E_\eta^+.
\]
This is the unique value at $t^*$ consistent with \cite[Definition~2.3(iv)]{KimMemoli2021}; it is not an additional hypothesis on the DBN,
since $\{k\Delta t\}$ has measure zero in $\BR_+$.

Now let $0 < \varepsilon < \Delta t$.  We verify both inclusions.

\begin{enumerate}[(i)]
  \item $E_\eta(t^* - \varepsilon) \subset E_\eta(t^*)$: suppose
    $\{Y,X\} \in E_\eta(t^* - \varepsilon)$.  Since
    $t^* - \varepsilon \in ((k-1)\Delta t, k\Delta t)$, we have
    $\delta^{(k-1)}_{YX} > \eta$, and therefore
    \[
      \max\!\bigl(\delta^{(k-1)}_{YX}, \delta^{(k)}_{YX}\bigr) \;\geq\; \delta^{(k-1)}_{YX} > \eta,
    \]
    so $\{Y,X\} \in E_\eta(t^*)$.

  \item $E_\eta(t^* + \varepsilon) \subset E_\eta(t^*)$: suppose
    $\{Y,X\} \in E_\eta(t^* + \varepsilon)$.  Since
    $t^* + \varepsilon \in (k\Delta t, (k+1)\Delta t)$, we have
    $\delta^{(k)}_{YX} > \eta$, and therefore
    \[
      \max\!\bigl(\delta^{(k-1)}_{YX}, \delta^{(k)}_{YX}\bigr) \;\geq\; \delta^{(k)}_{YX} > \eta,
    \]
    so $\{Y,X\} \in E_\eta(t^*)$.
\end{enumerate}

Hence $E_\eta(t^* - \varepsilon) \subset E_\eta(t^*) \supset E_\eta(t^* + \varepsilon)$
for all $0 < \varepsilon < \Delta t$.

\medskip
All four conditions of Definition \ref{Def:Dynamic-graph} are satisfied, so
$\mathscr{G}_{\mathcal{X}}$ is a dynamic graph.
\end{proof}

\medskip
\begin{remark}
The maximum convention at critical times has a natural interpretation:
at the moment $t^* = k\Delta t$ when the DBN transitions from time slice
$k-1$ to slice $k$, the graph $\mathscr{G}_{\mathcal{X}}(t^*)$ flashes
all edges that are strong enough on \emph{either} side of the transition.
Edges that are too weak on both sides never appear.  
\end{remark}


If we consider the fact that $\forall t \in \mathbb{R}_+$ $\mathcal{X_{\to}}^{(t)} = \mathcal{X}$, we actually obtain something called a \textit{saturated} dynamic graph. Moreover, this means that for a DG $\mathscr{G}_{X}= ( V_X(\cdot), E_X(\cdot))$, $V_X(\cdot) \equiv X$. If edges were not allowed to disappear, and if we did not introduce the function $E_\eta(\cdot)$, we would end up with what is called a \textit{constant} DG. A constant DG is a saturated DG, but where $E_X(\cdot)$ is a constant function. A constant DG is generally uninteresting, or trivial as $\textnormal{crit}(\mathscr{G}_X)= \emptyset$. This exact point gives rise to the function $E_\eta(\cdot)$.

\subsection{The Formigram of A Dynamic Bayesian Graph}\label{Sec:formigram}

With the construction of a dynamic Bayesian graph DBG, we can now apply methods from \cite{KimMemoli2021}. First, we will define the formigram of a DBG

\subsubsection{Definition and Remarks}

\begin{definition}[Formigram of A Dynamic Bayesian Graph]
    Take a DBG $\mathscr{G}_\mathcal{X}=(\mathcal{X}_\to^{(\cdot)},E_\eta(\cdot))$. We can apply the path component functor, $\pi_0$, to our DBG and get the function $\pi_0 (\mathscr{G}_\mathcal{X}): \mathbb{R}_+ \to \mathcal{P}(\mathcal{X})$ defined by $\pi_0 (\mathscr{G}_\mathcal{X})(t) = \pi_0(\mathscr{G}_\mathcal{X}(t))$ for $t \in \mathbb{R}_+$. Since a DBG is a DG, we say that $\pi_0(\mathscr{G}_\mathcal{X}(t))$ is a formigram by Proposition \ref{Prop:DG-to-formigram}.
\end{definition}

If we consider the fact that a dynamic Bayesian graph is really just a saturated dynamic graph, we can see what this implies for the formigram of a dynamic Bayesian graph. Since $\mathcal{X}_\to^{(t)} \equiv \mathcal{X}$ for all $t \in \mathbb{R}_+$, it becomes apparent that $\pi_0(\mathscr{G}_\mathcal{X}(t))$ will always be in $\mathcal{P}(\mathcal{X})$. This means that $\pi_0(\mathscr{G}_\mathcal{X})$ is what is called a \textit{saturated} formigram. A saturated formigram is a formigram $\theta_X : \mathbb{R} \to \mathcal{P}^{\textnormal{sub}}(X)$, where its image is  $\mathcal{P}(X) \subset \mathcal{P}^{\textnormal{sub}}(X)$.

\subsubsection{Examinations}

In addition to creating the notion of a DBG's formigram, it is important to examine the significance of merging and disbanding events in a DBG.

\begin{definition}[Merging, And Disbanding Events of A DBG]
    Consider the formigram of a DBG over $\mathcal{X}$, $\pi_0(\mathscr{G}_{\mathcal{X}})$. If $t^* \in \textnormal{crit}(\pi_0(\mathscr{G}_\mathcal{X}))$ and we have a sufficiently small $\varepsilon > 0$ where $\{ t^* \} = [t^* - \varepsilon, t^* + \varepsilon] \cap \textnormal{crit}(\pi_0(\mathscr{G}_\mathcal{X}))$ then
    \begin{enumerate}[(i)]
        \item a \textit{merging event} occurs at $t^* \in \mathbb{R}_+$ if we have two blocks $A,B \in \pi_0(\mathscr{G}_\mathcal{X})(t^* - \varepsilon)$ and a block $C \in \pi(\mathscr{G_{\mathcal{X}}})(t^*)$  where $A \cup B \subset C$,
        \item a \textit{disbanding event} occurs at $t^* \in \mathbb{R}_+$ if we have two blocks $A,B \in \pi_0(\mathscr{G}_\mathcal{X})(t^* + \varepsilon)$ and a block $C \in \pi(\mathscr{G_{\mathcal{X}}})(t^*)$  where $A \cup B \subset C$.
    \end{enumerate}
    Additionally, since a DBG is saturated, birth and death events do not occur.
\end{definition}

Considering our edge set $E_\eta(\cdot)$ and total variation distance edge strength function $\delta_{X'Y}^{(\cdot)}$ between two random variables $X', Y$, we can intuitively understand how a merging event will occur. Let us have two distinct sequences of connected nodes $X = X_0 , \cdots , X_n = X'$ and $Y = Y_0, \cdots, Y_n = Y'$ at time $t-\varepsilon$. Since these two sequences are distinct, this must mean that $\delta_{X'Y}^{(t-\varepsilon)}\leq \eta$. However, if at time $t$ we have that $\delta_{X'Y}^{(t-\varepsilon)} > \eta$, we will have a merging event as the two sequences will merge to become $X = X_0, \cdots, X_n = X' , Y = Y_0, \cdots, Y_n = Y'$. A similar result follows for disbanding events.

\subsection{Barcode of A Dynamic Bayesian Graph}

Much of the content of this section is tightly linked with information previously established in section \ref{Sec:Kim-Memoli-summary}. Hence, this section will be brief, mainly establishing the main results with some commentary.

As established in Section \ref{Sec:Zigzag}, the path component functor $\pi_0$ allows us to define a formigram associated with a DBG, $\pi_0 (\mathscr{G}_\mathcal{X})$. Since $\pi_0 (\mathscr{G}_\mathcal{X})$ is a formigram over the set $\mathcal{X}$, we are able to use Definition \ref{Def:Barcode-formigram} to obtain a barcode from a DBG

\begin{definition}[Barcode of A Dynamic Bayesian Graph]
    Let $\pi_0(\mathscr{G}_\mathcal{X})$ be a formigram of a DBG over the set $\mathcal{X}$. We say the multiset $\textnormal{dgm}(\mathscr{G}_\mathcal{X}) \coloneq \Psi_{K} \circ \textnormal{dgm}_K(\mathbb{V}_{\mathscr{G}_\mathcal{X}})$ defines the \textbf{barcode of a dynamic Bayesian graph}. 
\end{definition}

\subsection{Stability Results for Dynamic Bayesian Graphs}

This paper establishes stability results for DBGs. 

The relevant background for this is Section \ref{Sec:stability}. However, we must first introduce some additional concepts and definitions.

\begin{definition}[Time-Interlevel Smoothing of DBGs]
    If we consider the DBG $\mathscr{G}_{\mathcal{X}} = (\mathcal{X}_\to^{(\cdot)}, E_\eta(\cdot) )$ over $\mathcal{X}$, we can define the two following things
    \begin{enumerate}[(i)]
        \item Let $I \subset \mathbb{R}_+$ be and interval. We can define the following
        \begin{align}
            \bigcup_{I}\mathscr{G}_\mathcal{X} \coloneq \left ( \bigcup_{t \in I} \mathcal{X}_\to^{(t)}, \bigcup_{t \in I} E_\eta(t) \right ) = \left ( \mathcal{X}, \bigcup_{t \in I} E_\eta(t) \right )
        \end{align}
        \item The $\varepsilon-$smoothing of $\mathscr{G}_\mathcal{X}$, $S_\varepsilon \mathscr{G}_\mathcal{X}$, is defined in the following manner. Let $\varepsilon \geq 0$:
        \begin{align}
            S_\varepsilon \mathscr{G}_\mathcal{X}(t) = \bigcup_{[t]^\varepsilon} \mathscr{G}_\mathcal{X}
        \end{align}
        for $t \in \mathbb{R}$.
    \end{enumerate}
\end{definition}

\begin{proposition}
    Since the DBG $\mathscr{G}_\mathcal{X}$ is also a DG, by Propositions \ref{Prop:dynamic-Bayesian-dynamic-graph} and \ref{Prop:smoothing-DG}, $S_\varepsilon \mathscr{G}_\mathcal{X}$ fulfills the conditions for a dynamic graph. 
\end{proposition}

Since it has been established that $\mathscr{G}_\mathcal{X}$ and $S_\varepsilon \mathscr{G}_\mathcal{X}$ are both DGs, the following Theorem follows from Theorem \ref{Thm:pseudo-metric}.

\begin{theorem}\label{Thm:pseudometric}
    The function $d_1^{\textnormal{dynG}}$ is an extended pseudo metric on DBGs.
\end{theorem}

\begin{result}[The Stability Result for DBGs]
    Since DBGs are DGs, DBGs are mapped as formigrams in a way that is stable using the path component functor $\pi_0$. 
\end{result}

\begin{proposition}\label{Prop:distance}
    Let $\pi_0(\mathscr{G}_\mathcal{X})$ be a formigram of a DBG over the set of random variables $\mathcal{X}$. Since $\pi_0(\mathscr{G}_\mathcal{X})$ is a formigram, then for $\varepsilon \geq 0$,
    \begin{align}
        d_B(\textnormal{dgm}(S_\varepsilon \pi_0(\mathscr{G}_\mathcal{X}),\textnormal{dgm}(\pi_0(\mathscr{G}_\mathcal{X}))) \leq \varepsilon
    \end{align}
\end{proposition}

\section{Existing Work in Clustering BNs}

\subsection{Applications of BNs}

It is widely known that DBNs are a considerably useful tool in the arsenal of many different domains of science. One can consider medicine, biology, sociology, and engineering as some fields that utilize DBNs. However, it is often the case that DBNs can lend themselves to an intuitive way of understanding dependence structures between nodes, but still fall behind in terms of computation. The inference and learning of DBNs often lends itself to computational difficulties that tend towards NP-hard. For this reason, algorithms for \textit{Bayesian network clustering} have been used to support inference on marginal distributions, see \cite{wu2025optimal}. Existing clustering methods use various algorithmic contexts to support inference on DBNs. Some existing clustering algorithms include work by Albrecht and Ramamoorthy on clustering for selective filtering in a Dynamic Bayesian Network, \cite{Albrecht-Ramamoorthy}, and the DCMAP algorithm developed by \cite{wu2025optimal}. 
However, many of these algorithms lack a rigorous method for global approximation of clusters and how these clusters impact computation.

\subsection{Inference on A BN}

Inference on a Bayesian Network involves computing the marginal distribution on each variable $X \in \mathcal{X}$ given some evidence $\mathbf{E} = \{ X_1 = \overline{x_1}, X_2 = x_2, \cdots\}$. This marginal distribution is given by $P(X \mid \mathbf{E})$. Moreover, this distribution is calculated by marginalizing all other variables:
\begin{align}
    P(X \mid \mathbf{E}) &= \frac{P(X, \mathbf{E})}{P(\mathbf{E})} \\
    P(X,\mathbf{E}) &= \sum_{X_i \in \mathcal{X}-\{ X \} }  \prod_{i=1}^{n} P(X_i \mid \text{Pa}^G_{X_i}) \delta(X_i).
\end{align}
We can remember that $P(X_i \mid \text{Pa}^G_{X_i})$ represents the probability of $X_i$ given its parents. Furthermore, we say that $\delta(X_i)$ represents the evidence for variable $X_i$, and $\mathcal{X} - \{ X \}$ is the set difference where $X$ is removed from $\mathcal{X}$. An important observation is that the joint distribution $P(\mathcal{X}) = \prod_{i=1}^{n} P(X_i \mid \text{Pa}^G_{X_i})$ has order 
$2^{|\mathcal{X}|}$. 
So the number of entries in the joint table grows exponentially as $2^n$. We use marginalization to essentially remove these dimensions in the joint distribution. This is done by summing over all states $x_{i1}, x_{i2}, \cdots$ for every $X_i \in \mathcal X$. This is exactly what $\sum_{X_i \in \mathcal X - \{ X \}}$ represents. Then the term $P(\mathbf{E})$ is used to normalize the distribution. These results are further described in  \cite{wu2025optimal}.

\subsection{Example: Clustering Using DCMAP}

The following sections will summarize examples seen in \cite{wu2025optimal} related to clustering according to Dependent Cluster MAPping (DCMAP). Consider a Bayesian network $\mathcal{B}= (G,P)$. with the following structure. We will define our node set to be $\mathcal{X} = \{ X_1, X_2, X_3, X_4 , X_5, X_6, X_7\}$. We will also say that each of our random variables is a binary random variable, $X_i = \overline{x_i}$ or $X_i = x_i$, with dependence defined as in Figure \ref{Fig:Bayesian-network}.

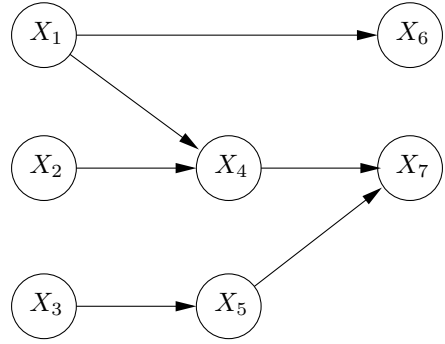
\begin{figure}[H]
    \centering
        \begin{tikzpicture}[x=0.65pt,y=0.65pt,yscale=-1,xscale=1]
        
        \draw    (330.75,220.75) -- (417.92,153.22) ;
        \draw [shift={(419.5,152)}, rotate = 142.24] [fill={rgb, 255:red, 0; green, 0; blue, 0 }  ][line width=0.08]  [draw opacity=0] (12,-3) -- (0,0) -- (12,3) -- cycle    ;
        \draw    (223.75,220.75) -- (310,220.75) ;
        \draw [shift={(312,220.75)}, rotate = 180] [fill={rgb, 255:red, 0; green, 0; blue, 0 }  ][line width=0.08]  [draw opacity=0] (12,-3) -- (0,0) -- (12,3) -- cycle    ;
        \draw    (330.75,140.75) -- (417,140.75) ;
        \draw [shift={(419,140.75)}, rotate = 180] [fill={rgb, 255:red, 0; green, 0; blue, 0 }  ][line width=0.08]  [draw opacity=0] (12,-3) -- (0,0) -- (12,3) -- cycle    ;
        \draw    (223.75,140.75) -- (310,140.75) ;
        \draw [shift={(312,140.75)}, rotate = 180] [fill={rgb, 255:red, 0; green, 0; blue, 0 }  ][line width=0.08]  [draw opacity=0] (12,-3) -- (0,0) -- (12,3) -- cycle    ;
        \draw    (223.75,63.75) -- (311.89,128.81) ;
        \draw [shift={(313.5,130)}, rotate = 216.43] [fill={rgb, 255:red, 0; green, 0; blue, 0 }  ][line width=0.08]  [draw opacity=0] (12,-3) -- (0,0) -- (12,3) -- cycle    ;
        \draw    (223.75,63.75) -- (415,63.75) ;
        \draw [shift={(417,63.75)}, rotate = 180] [fill={rgb, 255:red, 0; green, 0; blue, 0 }  ][line width=0.08]  [draw opacity=0] (12,-3) -- (0,0) -- (12,3) -- cycle    ;
        \draw  [fill={rgb, 255:red, 255; green, 255; blue, 255 }  ,fill opacity=1 ] (205,63.75) .. controls (205,53.39) and (213.39,45) .. (223.75,45) .. controls (234.11,45) and (242.5,53.39) .. (242.5,63.75) .. controls (242.5,74.11) and (234.11,82.5) .. (223.75,82.5) .. controls (213.39,82.5) and (205,74.11) .. (205,63.75) -- cycle ;
        \draw  [fill={rgb, 255:red, 255; green, 255; blue, 255 }  ,fill opacity=1 ] (205,140.75) .. controls (205,130.39) and (213.39,122) .. (223.75,122) .. controls (234.11,122) and (242.5,130.39) .. (242.5,140.75) .. controls (242.5,151.11) and (234.11,159.5) .. (223.75,159.5) .. controls (213.39,159.5) and (205,151.11) .. (205,140.75) -- cycle ;
        \draw  [fill={rgb, 255:red, 255; green, 255; blue, 255 }  ,fill opacity=1 ] (205,220.75) .. controls (205,210.39) and (213.39,202) .. (223.75,202) .. controls (234.11,202) and (242.5,210.39) .. (242.5,220.75) .. controls (242.5,231.11) and (234.11,239.5) .. (223.75,239.5) .. controls (213.39,239.5) and (205,231.11) .. (205,220.75) -- cycle ;
        \draw  [fill={rgb, 255:red, 255; green, 255; blue, 255 }  ,fill opacity=1 ] (312,140.75) .. controls (312,130.39) and (320.39,122) .. (330.75,122) .. controls (341.11,122) and (349.5,130.39) .. (349.5,140.75) .. controls (349.5,151.11) and (341.11,159.5) .. (330.75,159.5) .. controls (320.39,159.5) and (312,151.11) .. (312,140.75) -- cycle ;
        \draw  [fill={rgb, 255:red, 255; green, 255; blue, 255 }  ,fill opacity=1 ] (312,220.75) .. controls (312,210.39) and (320.39,202) .. (330.75,202) .. controls (341.11,202) and (349.5,210.39) .. (349.5,220.75) .. controls (349.5,231.11) and (341.11,239.5) .. (330.75,239.5) .. controls (320.39,239.5) and (312,231.11) .. (312,220.75) -- cycle ;
        \draw   (417,63.75) .. controls (417,53.39) and (425.39,45) .. (435.75,45) .. controls (446.11,45) and (454.5,53.39) .. (454.5,63.75) .. controls (454.5,74.11) and (446.11,82.5) .. (435.75,82.5) .. controls (425.39,82.5) and (417,74.11) .. (417,63.75) -- cycle ;
        \draw   (417,140.75) .. controls (417,130.39) and (425.39,122) .. (435.75,122) .. controls (446.11,122) and (454.5,130.39) .. (454.5,140.75) .. controls (454.5,151.11) and (446.11,159.5) .. (435.75,159.5) .. controls (425.39,159.5) and (417,151.11) .. (417,140.75) -- cycle ;
        
        \draw (214,54.4) node [anchor=north west][inner sep=0.75pt]    {$X_{1}$};
        \draw (214,131.4) node [anchor=north west][inner sep=0.75pt]    {$X_{2}$};
        \draw (214,211.4) node [anchor=north west][inner sep=0.75pt]    {$X_{3}$};
        \draw (321,131.4) node [anchor=north west][inner sep=0.75pt]    {$X_{4}$};
        \draw (321,211.4) node [anchor=north west][inner sep=0.75pt]    {$X_{5}$};
        \draw (426,54.4) node [anchor=north west][inner sep=0.75pt]    {$X_{6}$};
        \draw (426,131.4) node [anchor=north west][inner sep=0.75pt]    {$X_{7}$};

    \end{tikzpicture}
    \caption{A Bayesian network}
    \label{Fig:Bayesian-network}
\end{figure}


The goal with this Bayesian network is to perform inference. This means, given some evidence $\mathbf{E}$, we would like to be able to find the prior distribution $P(X_i, \mathbf{E}), \forall X_i \in \mathcal{X}$. First we will take notice of the CPT of each of the nodes in the network. The nodes $X_1, X_2, X_3$ each have a CPT of the form $P(X_i) = [p_{x_i}, p_{\overline{x_i}}]^T$ where $p_{x_i}$ represents the probability of $X_i =x_i$. Furthermore, since these nodes have no parents, they have a smaller CPT. Nodes like $X_5, X_7$ have a CPT of the form
\begin{align}
    P(X_i \mid X_j) = 
    \left [ 
    \begin{array}{cc}
        p_{x_ix_j} & p_{x_i \overline{x_j}} \\
        p_{\overline{x_i}x_j} & p_{\overline{x_i}\overline{x_j}}
    \end{array}
    \right ].
\end{align}
 
We will also note that $p_{x_ix_j}$ is the probability of $X_i = x_i$ given $X_j = x_j$. And nodes like $X_4, X_5$ have a CPT of the form
\begin{align}
    P(X_i \mid X_j, X_k) = 
    \left [ 
    \begin{array}{cccc}
        p_{x_i x_j x_k} & p_{x_i x_j \overline{x_k}} & p_{x_i \overline{x_j} x_k} & p_{x_i \overline{x_j}\overline{x_k}} \\
        p_{\overline{x_i} x_j x_k} & p_{\overline{x_i} x_j \overline{x_k}} & p_{\overline{x_i}\overline{x_j} x_k} & p_{\overline{x_i}\overline{x_j}\overline{x_k}}
    \end{array}
    \right ]
\end{align}
where $p_{x_ix_jx_k}$ represents the probability of $X_i = x_i$ given $X_j = x_j$ and $X_k = x_k$. We can see here in this example that more parents for a given $X_i$ results in higher dimensionality for its corresponding CPT. Moreover, if a particular variable is allowed to have many different states beyond just two, the dimensionality of our matrix increases even more. This will then result in an exponential growth in computation. Now if we use our equation for $P(X, \mathbf{E})$, we can obtain the marginal distribution for $P(X_6)$
\begin{align}
    P(X_6) = \sum_{\mathcal{X}-\{ X_6 \} }P(X_1)P(X_2)P(X_3)P(X_4 \mid X_1,X_2) P(X_5 \mid X_3 ) P(X_6 | X_1) P(X_7 | X_4, X_5).
\end{align}
However, computing $P(X_6)$ requires a substantial amount of calculation. Since each variable is binary, the full joint table for variables $\mathcal{X}$ is on the order of $\mathcal O(2^{n})$ where $n = |\mathcal X|$. 
According to \cite{wu2025optimal}, computing $P(X_6)$ naively requires $608$ products and $126$ sums. 
However, with DCMAP, we are able to reduce the computational complexity for determining the distribution. Consider the partition of our Bayesian network $\mathcal{B}$ in Figure \ref{Fig:clustering}.

\begin{figure}[H]
    \centering
    \begin{tikzpicture}[x=0.65pt,y=0.65pt,yscale=-1,xscale=1]
        
        \draw    (331.75,221.75) -- (418.92,154.22) ;
        \draw [shift={(420.5,153)}, rotate = 142.24] [fill={rgb, 255:red, 0; green, 0; blue, 0 }  ][line width=0.08]  [draw opacity=0] (12,-3) -- (0,0) -- (12,3) -- cycle    ;
        \draw    (224.75,221.75) -- (311,221.75) ;
        \draw [shift={(313,221.75)}, rotate = 180] [fill={rgb, 255:red, 0; green, 0; blue, 0 }  ][line width=0.08]  [draw opacity=0] (12,-3) -- (0,0) -- (12,3) -- cycle    ;
        \draw    (331.75,141.75) -- (418,141.75) ;
        \draw [shift={(420,141.75)}, rotate = 180] [fill={rgb, 255:red, 0; green, 0; blue, 0 }  ][line width=0.08]  [draw opacity=0] (12,-3) -- (0,0) -- (12,3) -- cycle    ;
        \draw    (224.75,141.75) -- (311,141.75) ;
        \draw [shift={(313,141.75)}, rotate = 180] [fill={rgb, 255:red, 0; green, 0; blue, 0 }  ][line width=0.08]  [draw opacity=0] (12,-3) -- (0,0) -- (12,3) -- cycle    ;
        \draw    (225.75,65.75) -- (313.89,130.81) ;
        \draw [shift={(315.5,132)}, rotate = 216.43] [fill={rgb, 255:red, 0; green, 0; blue, 0 }  ][line width=0.08]  [draw opacity=0] (12,-3) -- (0,0) -- (12,3) -- cycle    ;
        \draw    (224.75,64.75) -- (416,64.75) ;
        \draw [shift={(418,64.75)}, rotate = 180] [fill={rgb, 255:red, 0; green, 0; blue, 0 }  ][line width=0.08]  [draw opacity=0] (12,-3) -- (0,0) -- (12,3) -- cycle    ;
        \draw  [fill={rgb, 255:red, 255; green, 255; blue, 255 }  ,fill opacity=1 ] (206,64.75) .. controls (206,54.39) and (214.39,46) .. (224.75,46) .. controls (235.11,46) and (243.5,54.39) .. (243.5,64.75) .. controls (243.5,75.11) and (235.11,83.5) .. (224.75,83.5) .. controls (214.39,83.5) and (206,75.11) .. (206,64.75) -- cycle ;
        \draw  [fill={rgb, 255:red, 184; green, 233; blue, 134 }  ,fill opacity=1 ] (206,141.75) .. controls (206,131.39) and (214.39,123) .. (224.75,123) .. controls (235.11,123) and (243.5,131.39) .. (243.5,141.75) .. controls (243.5,152.11) and (235.11,160.5) .. (224.75,160.5) .. controls (214.39,160.5) and (206,152.11) .. (206,141.75) -- cycle ;
        \draw  [fill={rgb, 255:red, 128; green, 128; blue, 128 }  ,fill opacity=1 ] (206,221.75) .. controls (206,211.39) and (214.39,203) .. (224.75,203) .. controls (235.11,203) and (243.5,211.39) .. (243.5,221.75) .. controls (243.5,232.11) and (235.11,240.5) .. (224.75,240.5) .. controls (214.39,240.5) and (206,232.11) .. (206,221.75) -- cycle ;
        \draw  [fill={rgb, 255:red, 155; green, 155; blue, 155 }  ,fill opacity=1 ] (313,141.75) .. controls (313,131.39) and (321.39,123) .. (331.75,123) .. controls (342.11,123) and (350.5,131.39) .. (350.5,141.75) .. controls (350.5,152.11) and (342.11,160.5) .. (331.75,160.5) .. controls (321.39,160.5) and (313,152.11) .. (313,141.75) -- cycle ;
        \draw  [fill={rgb, 255:red, 74; green, 74; blue, 74 }  ,fill opacity=1 ] (313,221.75) .. controls (313,211.39) and (321.39,203) .. (331.75,203) .. controls (342.11,203) and (350.5,211.39) .. (350.5,221.75) .. controls (350.5,232.11) and (342.11,240.5) .. (331.75,240.5) .. controls (321.39,240.5) and (313,232.11) .. (313,221.75) -- cycle ;
        \draw   (418,64.75) .. controls (418,54.39) and (426.39,46) .. (436.75,46) .. controls (447.11,46) and (455.5,54.39) .. (455.5,64.75) .. controls (455.5,75.11) and (447.11,83.5) .. (436.75,83.5) .. controls (426.39,83.5) and (418,75.11) .. (418,64.75) -- cycle ;
        \draw  [fill={rgb, 255:red, 155; green, 155; blue, 155 }  ,fill opacity=1 ] (418,141.75) .. controls (418,131.39) and (426.39,123) .. (436.75,123) .. controls (447.11,123) and (455.5,131.39) .. (455.5,141.75) .. controls (455.5,152.11) and (447.11,160.5) .. (436.75,160.5) .. controls (426.39,160.5) and (418,152.11) .. (418,141.75) -- cycle ;
        \draw  [color={rgb, 255:red, 128; green, 128; blue, 128 }  ,draw opacity=1 ] (192,35) -- (256.75,35) -- (256.75,252) -- (192,252) -- cycle ;
        \draw  [color={rgb, 255:red, 155; green, 155; blue, 155 }  ,draw opacity=1 ] (299,114) -- (363.75,114) -- (363.75,252) -- (299,252) -- cycle ;
        \draw  [color={rgb, 255:red, 155; green, 155; blue, 155 }  ,draw opacity=1 ] (405,34) -- (469.75,34) -- (469.75,172) -- (405,172) -- cycle ;
        
        \draw (215,55.4) node [anchor=north west][inner sep=0.75pt]    {$X_{1}$};
        \draw (215,132.4) node [anchor=north west][inner sep=0.75pt]    {$X_{2}$};
        \draw (215,212.4) node [anchor=north west][inner sep=0.75pt]    {$X_{3}$};
        \draw (322,132.4) node [anchor=north west][inner sep=0.75pt]    {$X_{4}$};
        \draw (322,212.4) node [anchor=north west][inner sep=0.75pt]    {$X_{5}$};
        \draw (427,55.4) node [anchor=north west][inner sep=0.75pt]    {$X_{6}$};
        \draw (427,132.4) node [anchor=north west][inner sep=0.75pt]    {$X_{7}$};
    \end{tikzpicture}
\caption{Cluster and layers used in clustering algorithm}
\label{Fig:clustering}
\end{figure}
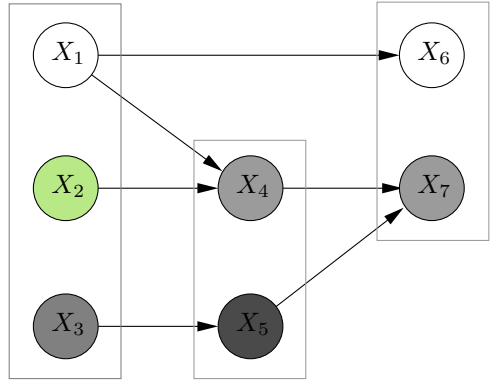


The shading in Figure \ref{Fig:clustering}  indicates the different clusters created by the DCMAP algorithm. Each box represents a ``layer" that is used in their clustering algorithm. The utility of this algorithm lies in its ability to reduce computational complexity when determining the distribution of $P(X_i)$. For instance, the distribution for $P(X_6)$ can now be represented in the following manner:
\begin{align}
\tiny{P(X_6)=\sum_{X_1}P(X_6\mid X_1)P(X_1)\left(\sum_{X_4,X_5,X_7}P(X_7\mid X_4,X_5)\left(\sum_{X_2}P(X_4\mid X_1,X_2)\left(\sum_{X_3}P(X_5\mid X_3)P(X_3)\right)\right)\right)}
\end{align}

For this particular specification of the distribution $P(X_6)$, clustering is able to significantly reduce the computational cost. The results from \cite{wu2025optimal} demonstrate the utility in seeing BNs as clusters for the sake of reducing computational cost.

\section{Interpreting Stable Distance Persistence Homology for Dynamic Bayesian Networks} \label{Sec:stable-distance-2}

An important question to ask with regard to SDPH is how it aligns with our given framework of DBNs. Moreover, how can we leverage clustering structure so that inference can be performed on a DBN more efficiently using SDPH. There are, in fact, several ways in which SDPH creates a new frame of reference in understanding DBNs.

\begin{note}[Key Motivation for Stable Distance Persistent Homology]
    We can list the different ways in which SDPH creates a useful framework for Bayesian Network clustering.
    \begin{enumerate}[(i)]
        \item Most clustering algorithms for BNs are heuristics that often use widely varying techniques to optimize computability for inference on a BN.
        \item These methods often fail to capture the time-evolving structure of DBNs across time. 
        \item Our method leverages the inherent structure of DBNs to track clustering signatures that emerge and disappear across time with respect to an edge strength that can represent a myriad of different relationships between nodes.
        \item The notion of edge strength is purposefully left abstract in order to give our framework adaptability in different arenas of inference goals.
        \item Using SDPH allows one to track the emergence and death of clustering signatures across time. With a strategic choice of edge strength, resultant clustering structures lend themselves to informative perspectives on DBNs.
        \item SDPH also lends itself to network comparison, meaning that a particular DBG $\mathscr{G}_{\mathcal{X}}$ can be compared to another DBG $\mathscr{G}_{\mathcal{Y}}$. Most obviously, this lends itself to stability in our persistence barcodes, but also more broadly defines a pseudo metric on the space of DBNs.
    \end{enumerate}
\end{note}

In considering these key points, it now becomes necessary to merge the language of SDPH with current frameworks for BN clustering. The following sections will further contextualize our results.

\subsection{Clusters}

\begin{definition}[Cluster Map]
    We say the \textbf{cluster map} $u(\cdot): \mathcal{X} \to \mathbb{N}$ is defined as $X \mapsto u(X)=k$ for $k \in \mathbb{N}$ and $X \in \mathcal{X}$, where $1 \leq u(X) \leq m$ is the index for $m$ unique clusters. Each cluster family associated with a particular $k \in [1,m]$ is denoted as $\mathbf{X}_k$.
\end{definition}

\begin{definition}[Static Path Component Clustering Map]
    For a BN $\mathcal{B}$ and its underlying Bayesian network structure $G = (\mathcal{X}, E_{\mathcal{X}})$, we define the \textbf{static path component clustering map} $u_{\pi_0}(\cdot): \mathcal{X} \to \mathbb{N}$ where for $X \in \mathcal X$, $X \to u_{\pi_0}(X) =k$, and $1 \leq u_{\pi_0}(X) \leq m$. Moreover, each clustering family $\mathbf{X}_k^{\pi_0}$ is defined by a relation where $X, X' \in \mathbf{X}_k^{\pi_0}$ if and only if $X \sim X'$. We also say that $X \sim X'$ if an only if there exists a finite sequence $X = X_1, X_2, \cdots,X_n = X'$ where $(X_i,X_{i+1}) \textnormal{ or } (X_{i+1}, X_i)\in E_\mathcal{X}$. 
\end{definition}

For any given BN, we require that there is no separation in the underlying Bayesian network structure. So, for this static path component clustering map will always result in one clustering family that is the entirety of the BN. However, if we are now to consider a DBN, we can create a clustering map with respect to its underlying DBG.

\begin{definition}[Time-Varying Path Component Clustering Map]
    For a DBN $\langle \mathcal{B}_0, \mathcal{B}_{\to} \rangle$ and its underlying DBG $\mathscr{G}_\mathcal{X}=(\mathcal{X}_\to^{(\cdot)},E_\eta(\cdot))$, we can remember the path component functor $\pi_0$ where $\pi_0(\mathscr{G}_\mathcal{X}): \mathbb{R}_+ \to \mathcal{P}(\mathcal{X})$ is defined by the formigram $\pi_0(\mathscr{{G}_\mathcal{X}})(t) = \pi_0(\mathscr{G}_\mathcal{X}(t))$ for $t \in \mathbb{R}_+$. We define the \textbf{time-varying path component clustering map} for some time $t \in \mathbb{R}_+$ as the function $u_{\pi_0}^t: \mathcal{X} \to \mathbb{N}$ defined by $X \mapsto u_{\pi_0}^t(X)= c$ where $1 \leq u_{\pi_0}^t(X) \leq |\pi_0(\mathscr{{G}_\mathcal{X}})(t)|$. Moreover, each clustering family at time $t$ is defined by the relation where the nodes $X, X'$ are in the cluster family $\mathbf{X}_c^{\pi_0}(t) \in \pi_0(\mathscr{{G}_\mathcal{X}})(t)$ if and only if $X \sim X'$. We also say that $X \sim X'$ if and only if there exists a finite sequence $X = X_1, X_2, \cdots,X_n = X'$ where $(X_i,X_{i+1}) \in E_\eta(t)$.
\end{definition}

\begin{remark}[Notation for Variables in $\mathbf{X}_c^{\pi_0}(t)$]
    Since we claim that $\mathbf{X}_c^{\pi_0}(t)$ is a cluster family for our DBN at time $t$ derived from the underlying DBG, we would expect to find that $\mathbf{X}_c^{\pi_0}(t)$ contains time indexed variables. However, It does not. For this reason we need redefine the clustering families for our time-varying path component clustering map.
\end{remark}

Now we can translate our clustering families back from the world of DBGs, and into the world of DBNs.

\begin{definition}[Time-Varying Clustering of A DBN]
    Consider the time intervals that define our DBN $[0, \Delta t), [\Delta t, 2 \Delta t), \cdots, [k \Delta t, (k+1)\Delta t)$. We can see that our instantiated random variables take on the form $ X_i^{(k)} \in  \mathcal X^{(k)}$ where $\mathcal X^{(k)}$ represents the random variables in the interval $[ k\Delta t, (k+1)\Delta t)$ for $k \in \mathbb{N}$, $0 \leq k \leq K$. We define the clustering family at integer time step $k$ as $\mathbf{X}_c(k) := \mathbf{X}_c^{\pi_0}(t^*)$, where $t^* = \tfrac{(2k+1)\Delta t}{2}$ is the midpoint of the interval $[k\Delta t, (k+1)\Delta t]$.
\end{definition}

\begin{remark}[Independence of Clustering to Inter-Time-Slice Edges]
    We may notice that clustering families for a particular moment in time $t$ does not reach across moments in time. This is a direct result of the fact that \textit{inter-time-slice edges} are not considered in the clustering model.
\end{remark}

\subsection{Clustering using SDPH: Example} 

Using these definitions, we can now work through an example. Consider the following Bayesian Network:

\begin{figure}[H]
    \begin{centering}
        \begin{tikzpicture}[x=0.75pt,y=0.75pt,yscale=-0.7,xscale=0.7]
        
        \draw    (405,51) -- (315.54,124.73) ;
        \draw [shift={(314,126)}, rotate = 320.51] [color={rgb, 255:red, 0; green, 0; blue, 0 }  ][line width=0.75]    (10.93,-3.29) .. controls (6.95,-1.4) and (3.31,-0.3) .. (0,0) .. controls (3.31,0.3) and (6.95,1.4) .. (10.93,3.29)   ;
        \draw    (405,51) -- (405,118) ;
        \draw [shift={(405,120)}, rotate = 270] [color={rgb, 255:red, 0; green, 0; blue, 0 }  ][line width=0.75]    (10.93,-3.29) .. controls (6.95,-1.4) and (3.31,-0.3) .. (0,0) .. controls (3.31,0.3) and (6.95,1.4) .. (10.93,3.29)   ;
        \draw    (405,51) -- (497.42,122.77) ;
        \draw [shift={(499,124)}, rotate = 217.83] [color={rgb, 255:red, 0; green, 0; blue, 0 }  ][line width=0.75]    (10.93,-3.29) .. controls (6.95,-1.4) and (3.31,-0.3) .. (0,0) .. controls (3.31,0.3) and (6.95,1.4) .. (10.93,3.29)   ;
        \draw    (299,140) -- (171.94,170.53) ;
        \draw [shift={(170,171)}, rotate = 346.49] [color={rgb, 255:red, 0; green, 0; blue, 0 }  ][line width=0.75]    (10.93,-3.29) .. controls (6.95,-1.4) and (3.31,-0.3) .. (0,0) .. controls (3.31,0.3) and (6.95,1.4) .. (10.93,3.29)   ;
        \draw    (405,140) -- (361.9,225.22) ;
        \draw [shift={(361,227)}, rotate = 296.83] [color={rgb, 255:red, 0; green, 0; blue, 0 }  ][line width=0.75]    (10.93,-3.29) .. controls (6.95,-1.4) and (3.31,-0.3) .. (0,0) .. controls (3.31,0.3) and (6.95,1.4) .. (10.93,3.29)   ;
        \draw    (405,140) -- (452.05,227.24) ;
        \draw [shift={(453,229)}, rotate = 241.66] [color={rgb, 255:red, 0; green, 0; blue, 0 }  ][line width=0.75]    (10.93,-3.29) .. controls (6.95,-1.4) and (3.31,-0.3) .. (0,0) .. controls (3.31,0.3) and (6.95,1.4) .. (10.93,3.29)   ;
        \draw    (152,179) -- (218.42,230.77) ;
        \draw [shift={(220,232)}, rotate = 217.93] [color={rgb, 255:red, 0; green, 0; blue, 0 }  ][line width=0.75]    (10.93,-3.29) .. controls (6.95,-1.4) and (3.31,-0.3) .. (0,0) .. controls (3.31,0.3) and (6.95,1.4) .. (10.93,3.29)   ;
        \draw  [fill={rgb, 255:red, 255; green, 255; blue, 255 }  ,fill opacity=1 ] (385,51) .. controls (385,39.95) and (393.95,31) .. (405,31) .. controls (416.05,31) and (425,39.95) .. (425,51) .. controls (425,62.05) and (416.05,71) .. (405,71) .. controls (393.95,71) and (385,62.05) .. (385,51) -- cycle ;
        \draw  [fill={rgb, 255:red, 255; green, 255; blue, 255 }  ,fill opacity=1 ] (491,140) .. controls (491,128.95) and (499.95,120) .. (511,120) .. controls (522.05,120) and (531,128.95) .. (531,140) .. controls (531,151.05) and (522.05,160) .. (511,160) .. controls (499.95,160) and (491,151.05) .. (491,140) -- cycle ;
        \draw  [fill={rgb, 255:red, 255; green, 255; blue, 255 }  ,fill opacity=1 ] (443,246) .. controls (443,234.95) and (451.95,226) .. (463,226) .. controls (474.05,226) and (483,234.95) .. (483,246) .. controls (483,257.05) and (474.05,266) .. (463,266) .. controls (451.95,266) and (443,257.05) .. (443,246) -- cycle ;
        \draw  [fill={rgb, 255:red, 255; green, 255; blue, 255 }  ,fill opacity=1 ] (385,140) .. controls (385,128.95) and (393.95,120) .. (405,120) .. controls (416.05,120) and (425,128.95) .. (425,140) .. controls (425,151.05) and (416.05,160) .. (405,160) .. controls (393.95,160) and (385,151.05) .. (385,140) -- cycle ;
        \draw  [fill={rgb, 255:red, 255; green, 255; blue, 255 }  ,fill opacity=1 ] (214,246) .. controls (214,234.95) and (222.95,226) .. (234,226) .. controls (245.05,226) and (254,234.95) .. (254,246) .. controls (254,257.05) and (245.05,266) .. (234,266) .. controls (222.95,266) and (214,257.05) .. (214,246) -- cycle ;
        \draw  [fill={rgb, 255:red, 255; green, 255; blue, 255 }  ,fill opacity=1 ] (332,246) .. controls (332,234.95) and (340.95,226) .. (352,226) .. controls (363.05,226) and (372,234.95) .. (372,246) .. controls (372,257.05) and (363.05,266) .. (352,266) .. controls (340.95,266) and (332,257.05) .. (332,246) -- cycle ;
        \draw  [fill={rgb, 255:red, 255; green, 255; blue, 255 }  ,fill opacity=1 ] (279,140) .. controls (279,128.95) and (287.95,120) .. (299,120) .. controls (310.05,120) and (319,128.95) .. (319,140) .. controls (319,151.05) and (310.05,160) .. (299,160) .. controls (287.95,160) and (279,151.05) .. (279,140) -- cycle ;
        \draw  [fill={rgb, 255:red, 255; green, 255; blue, 255 }  ,fill opacity=1 ] (132,179) .. controls (132,167.95) and (140.95,159) .. (152,159) .. controls (163.05,159) and (172,167.95) .. (172,179) .. controls (172,190.05) and (163.05,199) .. (152,199) .. controls (140.95,199) and (132,190.05) .. (132,179) -- cycle ;
        
        \draw (396,42.4) node [anchor=north west][inner sep=0.75pt]    {$X_{1}$};
        \draw (288,130.4) node [anchor=north west][inner sep=0.75pt]    {$X_{2}$};
        \draw (395,130.4) node [anchor=north west][inner sep=0.75pt]    {$X_{3}$};
        \draw (501,130.4) node [anchor=north west][inner sep=0.75pt]    {$X_{4}$};
        \draw (142,170.4) node [anchor=north west][inner sep=0.75pt]    {$X_{5}$};
        \draw (224,237.4) node [anchor=north west][inner sep=0.75pt]    {$X_{6}$};
        \draw (342,237.4) node [anchor=north west][inner sep=0.75pt]    {$X_{7}$};
        \draw (454,237.4) node [anchor=north west][inner sep=0.75pt]    {$X_{8}$};
        \end{tikzpicture}
        \caption{A Bayesian network structure over the random variables $\mathcal{X} = {X_1, \ldots, X_8}$. Directed edges encode conditional dependencies; the absence of an edge between two nodes encodes a conditional independence assumption.}
        \label{Fig:BN}
    \end{centering}
\end{figure}
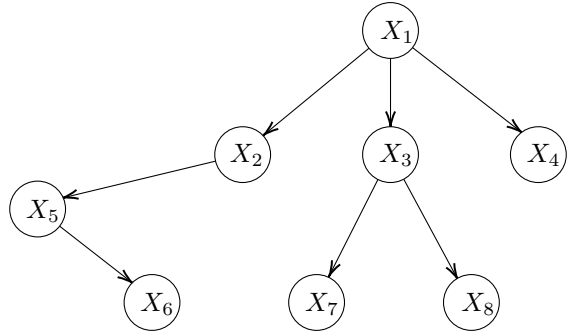


Furthermore, consider the following DBG created from instantiating the BN and observing the graph underlying the resulting DBN. In order to understand how some time change $\gamma$ may influence the inherent structure of a DBG and how this will influence clustering barcodes for a DBG, we can remember that a DBG does not maintain a notion of directionality between nodes. 
Since neither $X_2 \to X_3$ nor $X_3 \to X_2$ appears in G, the edge $\{X_2, X_3\}$ does not appear in the DBG (even though $X_2$ and $X_3$ are not marginally independent, being linked through their common parent $X_1$). 
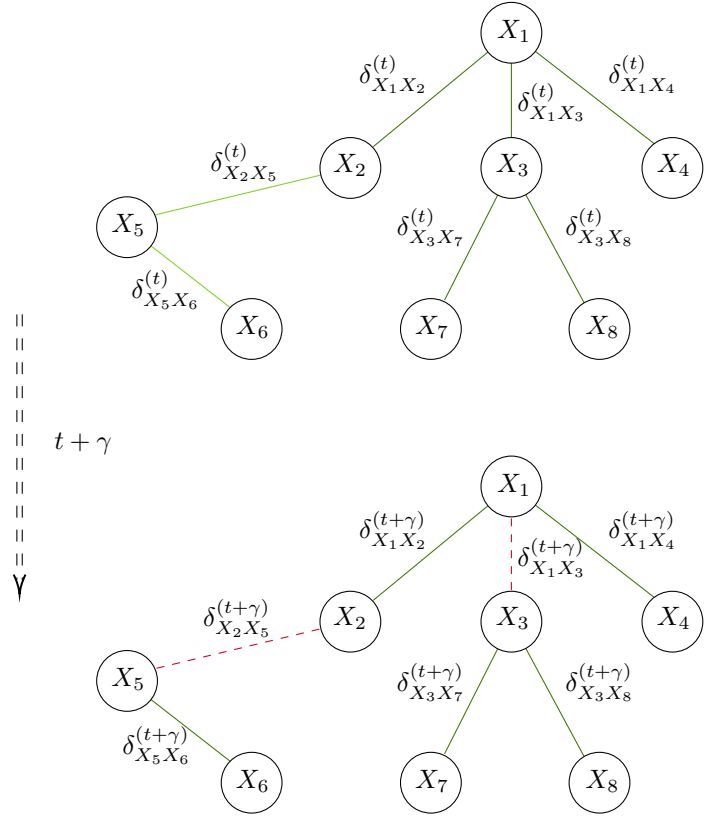
\begin{figure}[]
    \centering
    \begin{tikzpicture}[x=0.75pt,y=0.75pt,yscale=-0.9,xscale=0.9]
    
    \draw [color={rgb, 255:red, 65; green, 117; blue, 5 }  ,draw opacity=1 ]   (393.26,24.94) -- (316.18,88.48) ;
    \draw [color={rgb, 255:red, 65; green, 117; blue, 5 }  ,draw opacity=1 ]   (393.26,24.94) -- (393.26,83.39) ;
    \draw [color={rgb, 255:red, 65; green, 117; blue, 5 }  ,draw opacity=1 ]   (393.26,24.94) -- (472.89,86.78) ;
    \draw [color={rgb, 255:red, 126; green, 211; blue, 33 }  ,draw opacity=1 ]   (303.47,100.34) -- (194.19,126.6) ;
    \draw [color={rgb, 255:red, 65; green, 117; blue, 5 }  ,draw opacity=1 ]   (393.26,100.34) -- (355.99,174.04) ;
    \draw [color={rgb, 255:red, 65; green, 117; blue, 5 }  ,draw opacity=1 ]   (393.26,100.34) -- (433.92,175.73) ;
    \draw [color={rgb, 255:red, 126; green, 211; blue, 33 }  ,draw opacity=1 ]   (178.94,133.37) -- (236.55,178.27) ;
    \draw  [fill={rgb, 255:red, 255; green, 255; blue, 255 }  ,fill opacity=1 ] (376.32,24.94) .. controls (376.32,15.59) and (383.91,8) .. (393.26,8) .. controls (402.62,8) and (410.21,15.59) .. (410.21,24.94) .. controls (410.21,34.3) and (402.62,41.88) .. (393.26,41.88) .. controls (383.91,41.88) and (376.32,34.3) .. (376.32,24.94) -- cycle ;
    \draw  [fill={rgb, 255:red, 255; green, 255; blue, 255 }  ,fill opacity=1 ] (466.12,100.34) .. controls (466.12,90.98) and (473.7,83.39) .. (483.06,83.39) .. controls (492.41,83.39) and (500,90.98) .. (500,100.34) .. controls (500,109.69) and (492.41,117.28) .. (483.06,117.28) .. controls (473.7,117.28) and (466.12,109.69) .. (466.12,100.34) -- cycle ;
    \draw  [fill={rgb, 255:red, 255; green, 255; blue, 255 }  ,fill opacity=1 ] (425.45,190.13) .. controls (425.45,180.77) and (433.04,173.19) .. (442.4,173.19) .. controls (451.75,173.19) and (459.34,180.77) .. (459.34,190.13) .. controls (459.34,199.49) and (451.75,207.07) .. (442.4,207.07) .. controls (433.04,207.07) and (425.45,199.49) .. (425.45,190.13) -- cycle ;
    \draw  [fill={rgb, 255:red, 255; green, 255; blue, 255 }  ,fill opacity=1 ] (376.32,100.34) .. controls (376.32,90.98) and (383.91,83.39) .. (393.26,83.39) .. controls (402.62,83.39) and (410.21,90.98) .. (410.21,100.34) .. controls (410.21,109.69) and (402.62,117.28) .. (393.26,117.28) .. controls (383.91,117.28) and (376.32,109.69) .. (376.32,100.34) -- cycle ;
    \draw  [fill={rgb, 255:red, 255; green, 255; blue, 255 }  ,fill opacity=1 ] (231.46,190.13) .. controls (231.46,180.77) and (239.05,173.19) .. (248.41,173.19) .. controls (257.76,173.19) and (265.35,180.77) .. (265.35,190.13) .. controls (265.35,199.49) and (257.76,207.07) .. (248.41,207.07) .. controls (239.05,207.07) and (231.46,199.49) .. (231.46,190.13) -- cycle ;
    \draw  [fill={rgb, 255:red, 255; green, 255; blue, 255 }  ,fill opacity=1 ] (331.42,190.13) .. controls (331.42,180.77) and (339.01,173.19) .. (348.37,173.19) .. controls (357.72,173.19) and (365.31,180.77) .. (365.31,190.13) .. controls (365.31,199.49) and (357.72,207.07) .. (348.37,207.07) .. controls (339.01,207.07) and (331.42,199.49) .. (331.42,190.13) -- cycle ;
    \draw  [fill={rgb, 255:red, 255; green, 255; blue, 255 }  ,fill opacity=1 ] (286.53,100.34) .. controls (286.53,90.98) and (294.11,83.39) .. (303.47,83.39) .. controls (312.83,83.39) and (320.41,90.98) .. (320.41,100.34) .. controls (320.41,109.69) and (312.83,117.28) .. (303.47,117.28) .. controls (294.11,117.28) and (286.53,109.69) .. (286.53,100.34) -- cycle ;
    \draw  [fill={rgb, 255:red, 255; green, 255; blue, 255 }  ,fill opacity=1 ] (162,133.37) .. controls (162,124.02) and (169.59,116.43) .. (178.94,116.43) .. controls (188.3,116.43) and (195.88,124.02) .. (195.88,133.37) .. controls (195.88,142.73) and (188.3,150.32) .. (178.94,150.32) .. controls (169.59,150.32) and (162,142.73) .. (162,133.37) -- cycle ;
    \draw [color={rgb, 255:red, 65; green, 117; blue, 5 }  ,draw opacity=1 ]   (393.26,277.94) -- (316.18,341.48) ;
    \draw [color={rgb, 255:red, 208; green, 2; blue, 27 }  ,draw opacity=1, dashed ]   (393.26,277.94) -- (393.26,336.39) ;
    \draw [color={rgb, 255:red, 65; green, 117; blue, 5 }  ,draw opacity=1 ]   (393.26,277.94) -- (472.89,339.78) ;
    \draw [color={rgb, 255:red, 208; green, 2; blue, 27 }  ,draw opacity=1, dashed ]   (303.47,353.34) -- (194.19,379.6) ;
    \draw [color={rgb, 255:red, 65; green, 117; blue, 5 }  ,draw opacity=1 ]   (393.26,353.34) -- (355.99,427.04) ;
    \draw [color={rgb, 255:red, 65; green, 117; blue, 5 }  ,draw opacity=1 ]   (393.26,353.34) -- (433.92,428.73) ;
    \draw [color={rgb, 255:red, 65; green, 117; blue, 5 }  ,draw opacity=1 ]   (178.94,386.37) -- (236.55,431.27) ;
    \draw  [fill={rgb, 255:red, 255; green, 255; blue, 255 }  ,fill opacity=1 ] (376.32,277.94) .. controls (376.32,268.59) and (383.91,261) .. (393.26,261) .. controls (402.62,261) and (410.21,268.59) .. (410.21,277.94) .. controls (410.21,287.3) and (402.62,294.88) .. (393.26,294.88) .. controls (383.91,294.88) and (376.32,287.3) .. (376.32,277.94) -- cycle ;
    \draw  [fill={rgb, 255:red, 255; green, 255; blue, 255 }  ,fill opacity=1 ] (466.12,353.34) .. controls (466.12,343.98) and (473.7,336.39) .. (483.06,336.39) .. controls (492.41,336.39) and (500,343.98) .. (500,353.34) .. controls (500,362.69) and (492.41,370.28) .. (483.06,370.28) .. controls (473.7,370.28) and (466.12,362.69) .. (466.12,353.34) -- cycle ;
    \draw  [fill={rgb, 255:red, 255; green, 255; blue, 255 }  ,fill opacity=1 ] (425.45,443.13) .. controls (425.45,433.77) and (433.04,426.19) .. (442.4,426.19) .. controls (451.75,426.19) and (459.34,433.77) .. (459.34,443.13) .. controls (459.34,452.49) and (451.75,460.07) .. (442.4,460.07) .. controls (433.04,460.07) and (425.45,452.49) .. (425.45,443.13) -- cycle ;
    \draw  [fill={rgb, 255:red, 255; green, 255; blue, 255 }  ,fill opacity=1 ] (376.32,353.34) .. controls (376.32,343.98) and (383.91,336.39) .. (393.26,336.39) .. controls (402.62,336.39) and (410.21,343.98) .. (410.21,353.34) .. controls (410.21,362.69) and (402.62,370.28) .. (393.26,370.28) .. controls (383.91,370.28) and (376.32,362.69) .. (376.32,353.34) -- cycle ;
    \draw  [fill={rgb, 255:red, 255; green, 255; blue, 255 }  ,fill opacity=1 ] (231.46,443.13) .. controls (231.46,433.77) and (239.05,426.19) .. (248.41,426.19) .. controls (257.76,426.19) and (265.35,433.77) .. (265.35,443.13) .. controls (265.35,452.49) and (257.76,460.07) .. (248.41,460.07) .. controls (239.05,460.07) and (231.46,452.49) .. (231.46,443.13) -- cycle ;
    \draw  [fill={rgb, 255:red, 255; green, 255; blue, 255 }  ,fill opacity=1 ] (331.42,443.13) .. controls (331.42,433.77) and (339.01,426.19) .. (348.37,426.19) .. controls (357.72,426.19) and (365.31,433.77) .. (365.31,443.13) .. controls (365.31,452.49) and (357.72,460.07) .. (348.37,460.07) .. controls (339.01,460.07) and (331.42,452.49) .. (331.42,443.13) -- cycle ;
    \draw  [fill={rgb, 255:red, 255; green, 255; blue, 255 }  ,fill opacity=1 ] (286.53,353.34) .. controls (286.53,343.98) and (294.11,336.39) .. (303.47,336.39) .. controls (312.83,336.39) and (320.41,343.98) .. (320.41,353.34) .. controls (320.41,362.69) and (312.83,370.28) .. (303.47,370.28) .. controls (294.11,370.28) and (286.53,362.69) .. (286.53,353.34) -- cycle ;
    \draw  [fill={rgb, 255:red, 255; green, 255; blue, 255 }  ,fill opacity=1 ] (162,386.37) .. controls (162,377.02) and (169.59,369.43) .. (178.94,369.43) .. controls (188.3,369.43) and (195.88,377.02) .. (195.88,386.37) .. controls (195.88,395.73) and (188.3,403.32) .. (178.94,403.32) .. controls (169.59,403.32) and (162,395.73) .. (162,386.37) -- cycle ;
    \draw  [dash pattern={on 4.5pt off 4.5pt}]  (120.5,182) -- (120.5,330)(117.5,182) -- (117.5,330) ;
    \draw [shift={(119,338)}, rotate = 270] [color={rgb, 255:red, 0; green, 0; blue, 0 }  ][line width=0.75]    (10.93,-3.29) .. controls (6.95,-1.4) and (3.31,-0.3) .. (0,0) .. controls (3.31,0.3) and (6.95,1.4) .. (10.93,3.29)   ;
    
    \draw (383.96,16.34) node [anchor=north west][inner sep=0.75pt]    {$X_{1}$};
    \draw (292.47,90.89) node [anchor=north west][inner sep=0.75pt]    {$X_{2}$};
    \draw (383.11,90.89) node [anchor=north west][inner sep=0.75pt]    {$X_{3}$};
    \draw (472.9,90.89) node [anchor=north west][inner sep=0.75pt]    {$X_{4}$};
    \draw (168.79,124.77) node [anchor=north west][inner sep=0.75pt]    {$X_{5}$};
    \draw (238.25,181.53) node [anchor=north west][inner sep=0.75pt]    {$X_{6}$};
    \draw (338.21,181.53) node [anchor=north west][inner sep=0.75pt]    {$X_{7}$};
    \draw (433.09,181.53) node [anchor=north west][inner sep=0.75pt]    {$X_{8}$};
    \draw (383.96,269.34) node [anchor=north west][inner sep=0.75pt]    {$X_{1}$};
    \draw (292.47,343.89) node [anchor=north west][inner sep=0.75pt]    {$X_{2}$};
    \draw (383.11,343.89) node [anchor=north west][inner sep=0.75pt]    {$X_{3}$};
    \draw (472.9,343.89) node [anchor=north west][inner sep=0.75pt]    {$X_{4}$};
    \draw (168.79,377.77) node [anchor=north west][inner sep=0.75pt]    {$X_{5}$};
    \draw (238.25,434.53) node [anchor=north west][inner sep=0.75pt]    {$X_{6}$};
    \draw (338.21,434.53) node [anchor=north west][inner sep=0.75pt]    {$X_{7}$};
    \draw (433.09,434.53) node [anchor=north west][inner sep=0.75pt]    {$X_{8}$};
    \draw (307,36.07) node [anchor=north west][inner sep=0.75pt]    {$\delta _{X_{1} X_{2}}^{( t)}$};
    \draw (395.26,51.57) node [anchor=north west][inner sep=0.75pt]    {$\delta _{X_{1} X_{3}}^{( t)}$};
    \draw (446,36.07) node [anchor=north west][inner sep=0.75pt]    {$\delta _{X_{1} X_{4}}^{( t)}$};
    \draw (224,85.07) node [anchor=north west][inner sep=0.75pt]    {$\delta _{X_{2} X_{5}}^{( t)}$};
    \draw (180,156.07) node [anchor=north west][inner sep=0.75pt]    {$\delta _{X_{5} X_{6}}^{( t)}$};
    \draw (328,121.07) node [anchor=north west][inner sep=0.75pt]    {$\delta _{X_{3} X_{7}}^{( t)}$};
    \draw (422,121.07) node [anchor=north west][inner sep=0.75pt]    {$\delta _{X_{3} X_{8}}^{( t)}$};
    \draw (307,290.07) node [anchor=north west][inner sep=0.75pt]    {$\delta _{X_{1} X_{2}}^{( t+\gamma )}$};
    \draw (395.26,305.57) node [anchor=north west][inner sep=0.75pt]    {$\delta _{X_{1} X_{3}}^{( t+\gamma )}$};
    \draw (446,290.07) node [anchor=north west][inner sep=0.75pt]    {$\delta _{X_{1} X_{4}}^{( t+\gamma )}$};
    \draw (220,339.07) node [anchor=north west][inner sep=0.75pt]    {$\delta _{X_{2} X_{5}}^{( t+\gamma )}$};
    \draw (175,410.07) node [anchor=north west][inner sep=0.75pt]    {$\delta _{X_{5} X_{6}}^{( t+\gamma )}$};
    \draw (328,375.07) node [anchor=north west][inner sep=0.75pt]    {$\delta _{X_{3} X_{7}}^{( t+\gamma )}$};
    \draw (422,375.07) node [anchor=north west][inner sep=0.75pt]    {$\delta _{X_{3} X_{8}}^{( t+\gamma )}$};
    \draw (137,246.4) node [anchor=north west][inner sep=0.75pt]    {$t+\gamma $};
    \end{tikzpicture}
    \caption{Two snapshots of the DBG $\mathscr{G}_{\mathcal{X}}$ derived from the BN in Figure~\ref{Fig:BN}.
\emph{Top:} at time $t$, all edge strengths $\delta^{(t)}_{X_iX_j}$ exceed the threshold $\eta$,
so the graph is fully connected.
\emph{Bottom:} at time $t+\gamma$, the edge strengths $\delta^{(t+\gamma)}_{X_1X_3}$, $\delta^{(t+\gamma)}_{X_2X_5}$ (shown in red, dashed)
have dropped below $\eta$, disconnecting the graph.}

\end{figure}


After some time $\gamma$, we arrive at a different DBG. When we consider the formigram of a DBG at each of these particular times we get two different partitions of the nodes of our graphs, $\pi_0 (\mathscr{G}_\mathcal{X})(t)=\{ \{ X_1, X_2, X_3, X_4, X_5, X_6, X_7, X_8 \} \}$, $\pi_0(\mathscr{G}_\mathcal{X}(t + \gamma)) = \{ \{ X_1, X_2, X_4 \}, \{ X_5, X_6 \}, \{  X_3, X_7, X_8\} \}$. At time $t$, we see that our DBG is fully connected, all points have at least one path connecting each other. However, at time $t+ \gamma$, we see that there are three separate partitions of the DBG. From the path component
functor, we arrive at two different clustering sets of clustering families and each particular time. We can see that at time $t$ we get
\begin{align}
    \mathbf{X}_1^{\pi_0}(t) = \{ X_1, X_2, X_3, X_4, X_5, X_6, X_7, X_8 \}.
\end{align}
And at time $t + \gamma$ we get the following
\begin{align}
    \mathbf{X}_1^{\pi_0}(t+\gamma) &= \{ X_1, X_2, X_4 \} \\
    \mathbf{X}_2^{\pi_0}(t+\gamma) &= \{ X_5, X_6\}\\
    \mathbf{X}_3^{\pi_0}(t+\gamma) &= \{ X_3, X_7, X_8\}
\end{align}

Now, if we use our BN to recover the directionality of each of the edges, we see that we now have three separate sub-networks, shown in Figure \ref{fig:three-components}:
the sub-network on $\{X_1, X_2, X_4\}$, with directed edges inherited from~$G$;
the sub-network on $\{X_5, X_6\}$; and
the sub-network on $\{X_3, X_7, X_8\}$.

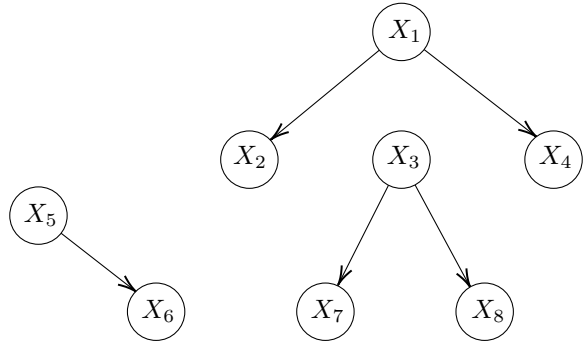
\begin{figure}[]
    \centering
    \begin{tikzpicture}[x=0.75pt,y=0.75pt,yscale=-0.85,xscale=0.85]
    
    \draw [color={rgb, 255:red, 0; green, 0; blue, 0 }  ,draw opacity=1 ]   (393.26,21.81) -- (317.72,84.07) ;
    \draw [shift={(316.18,85.35)}, rotate = 320.51] [color={rgb, 255:red, 0; green, 0; blue, 0 }  ,draw opacity=1 ][line width=0.75]    (10.93,-3.29) .. controls (6.95,-1.4) and (3.31,-0.3) .. (0,0) .. controls (3.31,0.3) and (6.95,1.4) .. (10.93,3.29)   ;
    \draw [color={rgb, 255:red, 0; green, 0; blue, 0 }  ,draw opacity=1 ]   (393.26,21.81) -- (471.31,82.42) ;
    \draw [shift={(472.89,83.65)}, rotate = 217.83] [color={rgb, 255:red, 0; green, 0; blue, 0 }  ,draw opacity=1 ][line width=0.75]    (10.93,-3.29) .. controls (6.95,-1.4) and (3.31,-0.3) .. (0,0) .. controls (3.31,0.3) and (6.95,1.4) .. (10.93,3.29)   ;
    \draw [color={rgb, 255:red, 0; green, 0; blue, 0 }  ,draw opacity=1 ]   (393.26,97.21) -- (356.89,169.12) ;
    \draw [shift={(355.99,170.9)}, rotate = 296.83] [color={rgb, 255:red, 0; green, 0; blue, 0 }  ,draw opacity=1 ][line width=0.75]    (10.93,-3.29) .. controls (6.95,-1.4) and (3.31,-0.3) .. (0,0) .. controls (3.31,0.3) and (6.95,1.4) .. (10.93,3.29)   ;
    \draw [color={rgb, 255:red, 0; green, 0; blue, 0 }  ,draw opacity=1 ]   (393.26,97.21) -- (432.98,170.84) ;
    \draw [shift={(433.92,172.6)}, rotate = 241.66] [color={rgb, 255:red, 0; green, 0; blue, 0 }  ,draw opacity=1 ][line width=0.75]    (10.93,-3.29) .. controls (6.95,-1.4) and (3.31,-0.3) .. (0,0) .. controls (3.31,0.3) and (6.95,1.4) .. (10.93,3.29)   ;
    \draw [color={rgb, 255:red, 0; green, 0; blue, 0 }  ,draw opacity=1 ]   (178.94,130.24) -- (234.97,173.91) ;
    \draw [shift={(236.55,175.14)}, rotate = 217.93] [color={rgb, 255:red, 0; green, 0; blue, 0 }  ,draw opacity=1 ][line width=0.75]    (10.93,-3.29) .. controls (6.95,-1.4) and (3.31,-0.3) .. (0,0) .. controls (3.31,0.3) and (6.95,1.4) .. (10.93,3.29)   ;
    \draw  [fill={rgb, 255:red, 255; green, 255; blue, 255 }  ,fill opacity=1 ] (376.32,21.81) .. controls (376.32,12.46) and (383.91,4.87) .. (393.26,4.87) .. controls (402.62,4.87) and (410.21,12.46) .. (410.21,21.81) .. controls (410.21,31.17) and (402.62,38.75) .. (393.26,38.75) .. controls (383.91,38.75) and (376.32,31.17) .. (376.32,21.81) -- cycle ;
    \draw  [fill={rgb, 255:red, 255; green, 255; blue, 255 }  ,fill opacity=1 ] (466.12,97.21) .. controls (466.12,87.85) and (473.7,80.26) .. (483.06,80.26) .. controls (492.41,80.26) and (500,87.85) .. (500,97.21) .. controls (500,106.56) and (492.41,114.15) .. (483.06,114.15) .. controls (473.7,114.15) and (466.12,106.56) .. (466.12,97.21) -- cycle ;
    \draw  [fill={rgb, 255:red, 255; green, 255; blue, 255 }  ,fill opacity=1 ] (425.45,187) .. controls (425.45,177.64) and (433.04,170.06) .. (442.4,170.06) .. controls (451.75,170.06) and (459.34,177.64) .. (459.34,187) .. controls (459.34,196.36) and (451.75,203.94) .. (442.4,203.94) .. controls (433.04,203.94) and (425.45,196.36) .. (425.45,187) -- cycle ;
    \draw  [fill={rgb, 255:red, 255; green, 255; blue, 255 }  ,fill opacity=1 ] (376.32,97.21) .. controls (376.32,87.85) and (383.91,80.26) .. (393.26,80.26) .. controls (402.62,80.26) and (410.21,87.85) .. (410.21,97.21) .. controls (410.21,106.56) and (402.62,114.15) .. (393.26,114.15) .. controls (383.91,114.15) and (376.32,106.56) .. (376.32,97.21) -- cycle ;
    \draw  [fill={rgb, 255:red, 255; green, 255; blue, 255 }  ,fill opacity=1 ] (231.46,187) .. controls (231.46,177.64) and (239.05,170.06) .. (248.41,170.06) .. controls (257.76,170.06) and (265.35,177.64) .. (265.35,187) .. controls (265.35,196.36) and (257.76,203.94) .. (248.41,203.94) .. controls (239.05,203.94) and (231.46,196.36) .. (231.46,187) -- cycle ;
    \draw  [fill={rgb, 255:red, 255; green, 255; blue, 255 }  ,fill opacity=1 ] (331.42,187) .. controls (331.42,177.64) and (339.01,170.06) .. (348.37,170.06) .. controls (357.72,170.06) and (365.31,177.64) .. (365.31,187) .. controls (365.31,196.36) and (357.72,203.94) .. (348.37,203.94) .. controls (339.01,203.94) and (331.42,196.36) .. (331.42,187) -- cycle ;
    \draw  [fill={rgb, 255:red, 255; green, 255; blue, 255 }  ,fill opacity=1 ] (286.53,97.21) .. controls (286.53,87.85) and (294.11,80.26) .. (303.47,80.26) .. controls (312.83,80.26) and (320.41,87.85) .. (320.41,97.21) .. controls (320.41,106.56) and (312.83,114.15) .. (303.47,114.15) .. controls (294.11,114.15) and (286.53,106.56) .. (286.53,97.21) -- cycle ;
    \draw  [fill={rgb, 255:red, 255; green, 255; blue, 255 }  ,fill opacity=1 ] (162,130.24) .. controls (162,120.89) and (169.59,113.3) .. (178.94,113.3) .. controls (188.3,113.3) and (195.88,120.89) .. (195.88,130.24) .. controls (195.88,139.6) and (188.3,147.19) .. (178.94,147.19) .. controls (169.59,147.19) and (162,139.6) .. (162,130.24) -- cycle ;
    
    \draw (383.96,13.21) node [anchor=north west][inner sep=0.75pt]    {$X_{1}$};
    \draw (292.47,87.76) node [anchor=north west][inner sep=0.75pt]    {$X_{2}$};
    \draw (383.11,87.76) node [anchor=north west][inner sep=0.75pt]    {$X_{3}$};
    \draw (472.9,87.76) node [anchor=north west][inner sep=0.75pt]    {$X_{4}$};
    \draw (168.79,121.64) node [anchor=north west][inner sep=0.75pt]    {$X_{5}$};
    \draw (238.25,178.4) node [anchor=north west][inner sep=0.75pt]    {$X_{6}$};
    \draw (338.21,178.4) node [anchor=north west][inner sep=0.75pt]    {$X_{7}$};
    \draw (433.09,178.4) node [anchor=north west][inner sep=0.75pt]    {$X_{8}$};

    \end{tikzpicture}
    \caption{The three connected components of $\mathscr{G}_{\mathcal{X}}$ at time $t + \gamma$,
obtained by removing the edges $\{X_1, X_3\}$ and $\{X_2, X_5\}$ from the fully connected graph.
The clusters are $\{X_1, X_2, X_4\}$, $\{X_5, X_6\}$, and $\{X_3, X_7, X_8\}$.
Directed edges are recovered from the underlying BN structure.}

    \label{fig:three-components}
\end{figure}

In the barcode of this particular DBG, we are able to obtain a unique kind of clustering based upon the edge strength assigned to our DBG.
At time $t$, the DBG is fully connected, so $\pi_0(\mathscr{G}_\mathcal{X})(t)$ consists of a single block $\{X_1,...,X_8\}$. This gives one bar in $\textnormal{dgm}(\mathscr{G}_\mathcal{X})$ alive at time $t$. After the disbanding event at time $t+\gamma$, three new bars are born in $\textnormal{dgm}(\mathscr{G}_\mathcal{X})$, corresponding to the clusters $\{X_1,X_2,X_4\}$, $\{X_5,X_6\}$, and $\{X_3,X_7,X_8\}$.

\section{The Utility of Partitioning a Dynamic Bayesian Network}

As previously discussed, there are several ways that we can dissect BNs. Previously, we discussed the utility of the DCMAP algorithm at reducing computational complexity in inference on a Bayesian network. Other algorithms, such as the jointree clustering described in \cite{LS-1988}, are able to effectively reduce computational complexity for inference in BNs. Furthermore, clustering can also be used to reduce space complexity in a BN by limiting the size of individual clusters within a BN, see \cite{Triplet}.

These methods are in slight contrast to ours. Our method utilizes the inherent topological structure of DBNs dependent on some graphical edge strength in order to provide time-dependent clustering. This method is different from current clustering methods in statistics that often rely on graph theory and optimization in order to arrive at clusters tailored to a specific purpose. A key aspect of our clustering algorithm is that we have distinct clusters at each particular time step. This is a key difference to methods such as the \textit{DCMAP algorithm} or the \textit{jointree algorithm}. 

\begin{note}[Insights from SDPH on DBNs]
    The practical insights of this paper are as follows:
    \begin{enumerate}[(i)]
        \item understand a DBN in terms of its inherent and useful graphical structure
        \item leverage the ambiguity of edge strength to allow for multiple different kinds of clustering,
        \item utilize a DBN's graphical structure to perform stable distance persistent homology via Zig-zag persistence,
        \item and use stability results in order to maintain legitimacy in signature output for real-world data.
    \end{enumerate}
\end{note}

\subsection{The Importance of Remembering a DBNs Graphical Structure}

One of the key advantages of BNs and DBNs is that they both maintain a graphical structure. There are two equally relevant and important perspectives on the underlying graph $G$ of a BN.
\begin{enumerate}[(i)]
    \item One perspective claims that the underlying graphical structure lends itself to a skeleton that encodes a compact factorization of the related joint distribution of the graph.
    \item Another perspective is that the underlying graphical structure is a compact way of being able to understand the conditional independence structure for a particular distribution, see \cite{Koller-Friedman}. 
\end{enumerate}
The difference between clustering such as the \textit{DCMAP algorithm} and the \textit{jointree algorithm} is that clustering is prioritized as a method to further factorize the underlying joint distribution of a BN, something more akin to the first perspective; while SDPH attempts to read the evolving conditional independence structure of a DBN and its associated distribution across time, something more akin to the second perspective. This distinction is important as it must be acknowledged that topological data analysis utilizes different schemes in order to obtain its clustering structure.

\subsection{Understanding Edge Strength as A Design Choice}

One of the most important aspects of using SDPH on a DBN is the generality associated with the edge strength function used to obtain stable signatures. In particular, we used $\delta_{XY}^{(t)}$ to represent the edge strength from $X$ to $Y$. For example, we used the total variation distance between each of these nodes along with the upper diameter of sub-stochastic matrices $P_{i \mid \textbf{x}}$ to encode an edge weight that represented the dependence that the child node $Y$ had on the parent node $X$. However, this is just one particular choice of edge strength. Within our model, as long as the DBG that is induced by some edge strength $\delta_{XY}^{(t)}$ satisfies the axioms for a DBG (self-loops, tameness, lifespan of vertices, and comparability), it is a viable candidate for SDPH. While these measures are not explored within our paper, it helps to consider the wide range for defining edge strength functions. An edge strength function could use any one of the following distances:
\begin{enumerate}[(i)]
    \item \textbf{Kullback--Leibler Distance:}
    This represents the distance between any two given statistical objects. Moreover, it measures the divergence between probability distributions:
    
    \begin{equation}
        D(P \| Q)
        =
        \int_{-\infty}^{\infty}
        p(x)\log\left(\frac{p(x)}{q(x)}\right)\,dx.
    \end{equation}

    However, it is also important to note that this distance is not symmetric and does require an algorithm for a symmetrization process \cite{Borges2020CRATES}.

    \item \textbf{Hellinger Distance:}
    This distance is essentially the closest analogue to the Euclidean distance for two discrete distributions $P$ and $Q$. The distance is given by the formula
    
    \begin{equation}
        HD(P,Q)
        =
        \frac{1}{\sqrt{2}}
        \sqrt{
        \sum_{i=1}^k
        \left(
        \sqrt{p(i)}-\sqrt{q(i)}
        \right)^2
        }.
    \end{equation}

    This distance is fairly common in statistical models and has uses in clustering statistical models \cite{Borges2020CRATES}.
    
    \item \textbf{Bhattacharyya Distance:}
    This distance is used heavily in data mining. In addition to this, it is commonly used for clustering purposes \cite{Borges2020CRATES}.
    \begin{align}
        BD(P,Q) = - \log \left ( \sum_{i =1 }^k \sqrt{p(i)q(i)} \right ).
    \end{align}
\end{enumerate}

\subsection{The Significance of Clustering Barcodes and Stability}

The key reason for performing SDPH on a DBN is the ability to capture clustering signatures within our DBN that evolve across time. Using SDPH, we are able to capture essential moments within our DBN where clustering structure shifts. These points are directly related to the \textit{critical points} associated with our DBG. These critical points lend themselves to deeper insights such as \textit{merging events} and \textit{disbanding events}, where, for instance, merging events represent moments in our DBN where clustering families $\mathbf{X}_k(t^* - \varepsilon)$ and $\mathbf{X}_j(t^* - \varepsilon)$ merge together to form a larger cluster $\mathbf{X}_i(t^*)$. This dynamic behavior can be clearly seen when observing the barcode of a DBG

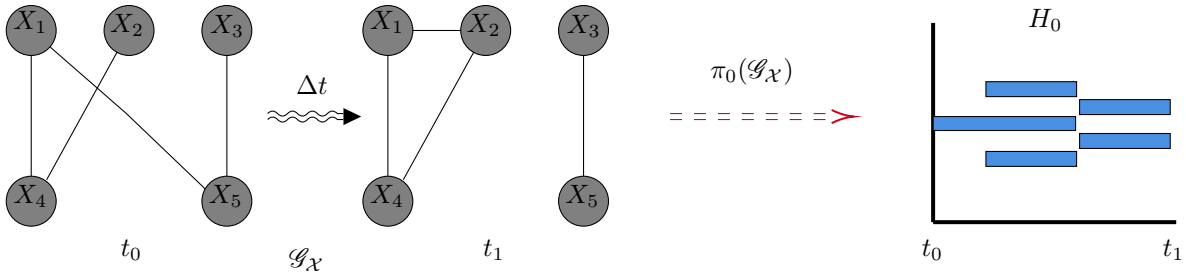
\begin{figure}[H]
    \centering
    \begin{tikzpicture}[x=0.75pt,y=0.75pt,yscale=-1,xscale=1]
        
        \draw    (93.46,33.8) -- (52.09,108.73) ;
        \draw    (44.36,46.33) -- (44.36,106.95) ;
        \draw    (142.56,46.33) -- (142.56,106.95) ;
        \draw    (44.36,33.8) -- (91.84,75.23) -- (131.93,113.55) ;
        \draw  [fill={rgb, 255:red, 128; green, 128; blue, 128 }  ,fill opacity=1 ] (31.83,33.8) .. controls (31.83,26.88) and (37.44,21.28) .. (44.36,21.28) .. controls (51.28,21.28) and (56.88,26.88) .. (56.88,33.8) .. controls (56.88,40.72) and (51.28,46.33) .. (44.36,46.33) .. controls (37.44,46.33) and (31.83,40.72) .. (31.83,33.8) -- cycle ;
        \draw  [fill={rgb, 255:red, 128; green, 128; blue, 128 }  ,fill opacity=1 ] (31.83,119.47) .. controls (31.83,112.56) and (37.44,106.95) .. (44.36,106.95) .. controls (51.28,106.95) and (56.88,112.56) .. (56.88,119.47) .. controls (56.88,126.39) and (51.28,132) .. (44.36,132) .. controls (37.44,132) and (31.83,126.39) .. (31.83,119.47) -- cycle ;
        \draw  [fill={rgb, 255:red, 128; green, 128; blue, 128 }  ,fill opacity=1 ] (130.03,119.47) .. controls (130.03,112.56) and (135.64,106.95) .. (142.56,106.95) .. controls (149.47,106.95) and (155.08,112.56) .. (155.08,119.47) .. controls (155.08,126.39) and (149.47,132) .. (142.56,132) .. controls (135.64,132) and (130.03,126.39) .. (130.03,119.47) -- cycle ;
        \draw  [fill={rgb, 255:red, 128; green, 128; blue, 128 }  ,fill opacity=1 ] (130.03,33.8) .. controls (130.03,26.88) and (135.64,21.28) .. (142.56,21.28) .. controls (149.47,21.28) and (155.08,26.88) .. (155.08,33.8) .. controls (155.08,40.72) and (149.47,46.33) .. (142.56,46.33) .. controls (135.64,46.33) and (130.03,40.72) .. (130.03,33.8) -- cycle ;
        \draw  [fill={rgb, 255:red, 128; green, 128; blue, 128 }  ,fill opacity=1 ] (80.93,33.8) .. controls (80.93,26.88) and (86.54,21.28) .. (93.46,21.28) .. controls (100.37,21.28) and (105.98,26.88) .. (105.98,33.8) .. controls (105.98,40.72) and (100.37,46.33) .. (93.46,46.33) .. controls (86.54,46.33) and (80.93,40.72) .. (80.93,33.8) -- cycle ;
        \draw    (272.46,33.8) -- (231.09,108.73) ;
        \draw    (223.36,46.33) -- (223.36,106.95) ;
        \draw    (321.56,46.33) -- (321.56,106.95) ;
        \draw  [fill={rgb, 255:red, 128; green, 128; blue, 128 }  ,fill opacity=1 ] (210.83,33.8) .. controls (210.83,26.88) and (216.44,21.28) .. (223.36,21.28) .. controls (230.28,21.28) and (235.88,26.88) .. (235.88,33.8) .. controls (235.88,40.72) and (230.28,46.33) .. (223.36,46.33) .. controls (216.44,46.33) and (210.83,40.72) .. (210.83,33.8) -- cycle ;
        \draw  [fill={rgb, 255:red, 128; green, 128; blue, 128 }  ,fill opacity=1 ] (210.83,119.47) .. controls (210.83,112.56) and (216.44,106.95) .. (223.36,106.95) .. controls (230.28,106.95) and (235.88,112.56) .. (235.88,119.47) .. controls (235.88,126.39) and (230.28,132) .. (223.36,132) .. controls (216.44,132) and (210.83,126.39) .. (210.83,119.47) -- cycle ;
        \draw  [fill={rgb, 255:red, 128; green, 128; blue, 128 }  ,fill opacity=1 ] (309.03,119.47) .. controls (309.03,112.56) and (314.64,106.95) .. (321.56,106.95) .. controls (328.47,106.95) and (334.08,112.56) .. (334.08,119.47) .. controls (334.08,126.39) and (328.47,132) .. (321.56,132) .. controls (314.64,132) and (309.03,126.39) .. (309.03,119.47) -- cycle ;
        \draw  [fill={rgb, 255:red, 128; green, 128; blue, 128 }  ,fill opacity=1 ] (309.03,33.8) .. controls (309.03,26.88) and (314.64,21.28) .. (321.56,21.28) .. controls (328.47,21.28) and (334.08,26.88) .. (334.08,33.8) .. controls (334.08,40.72) and (328.47,46.33) .. (321.56,46.33) .. controls (314.64,46.33) and (309.03,40.72) .. (309.03,33.8) -- cycle ;
        \draw  [fill={rgb, 255:red, 128; green, 128; blue, 128 }  ,fill opacity=1 ] (259.93,33.8) .. controls (259.93,26.88) and (265.54,21.28) .. (272.46,21.28) .. controls (279.37,21.28) and (284.98,26.88) .. (284.98,33.8) .. controls (284.98,40.72) and (279.37,46.33) .. (272.46,46.33) .. controls (265.54,46.33) and (259.93,40.72) .. (259.93,33.8) -- cycle ;
        \draw    (235.88,33.8) -- (259.93,33.8) ;
        \draw    (163,75.5) .. controls (164.67,73.83) and (166.33,73.83) .. (168,75.5) .. controls (169.67,77.17) and (171.33,77.17) .. (173,75.5) .. controls (174.67,73.83) and (176.33,73.83) .. (178,75.5) .. controls (179.67,77.17) and (181.33,77.17) .. (183,75.5) .. controls (184.67,73.83) and (186.33,73.83) .. (188,75.5) .. controls (189.67,77.17) and (191.33,77.17) .. (193,75.5) .. controls (194.67,73.83) and (196.33,73.83) .. (198,75.5) -- (201,75.5)(163,78.5) .. controls (164.67,76.83) and (166.33,76.83) .. (168,78.5) .. controls (169.67,80.17) and (171.33,80.17) .. (173,78.5) .. controls (174.67,76.83) and (176.33,76.83) .. (178,78.5) .. controls (179.67,80.17) and (181.33,80.17) .. (183,78.5) .. controls (184.67,76.83) and (186.33,76.83) .. (188,78.5) .. controls (189.67,80.17) and (191.33,80.17) .. (193,78.5) .. controls (194.67,76.83) and (196.33,76.83) .. (198,78.5) -- (201,78.5) ;
        \draw [shift={(210,77)}, rotate = 180] [fill={rgb, 255:red, 0; green, 0; blue, 0 }  ][line width=0.08]  [draw opacity=0] (8.93,-4.29) -- (0,0) -- (8.93,4.29) -- cycle    ;
        \draw [color={rgb, 255:red, 208; green, 2; blue, 27 }  ,draw opacity=1 ] [dash pattern={on 4.5pt off 4.5pt}]  (365,75.5) -- (449,75.5)(365,78.5) -- (449,78.5) ;
        \draw [shift={(457,77)}, rotate = 180] [color={rgb, 255:red, 208; green, 2; blue, 27 }  ,draw opacity=1 ][line width=0.75]    (10.93,-3.29) .. controls (6.95,-1.4) and (3.31,-0.3) .. (0,0) .. controls (3.31,0.3) and (6.95,1.4) .. (10.93,3.29)   ;
        \draw [line width=1.5]    (497,30) -- (497,130) ;
        \draw [line width=1.5]    (497,130) -- (619,130) ;
        \draw  [fill={rgb, 255:red, 74; green, 144; blue, 226 }  ,fill opacity=1 ] (497,77) -- (568.5,77) -- (568.5,84) -- (497,84) -- cycle ;
        \draw  [fill={rgb, 255:red, 74; green, 144; blue, 226 }  ,fill opacity=1 ] (523.5,59.5) -- (569,59.5) -- (569,67) -- (523.5,67) -- cycle ;
        \draw  [fill={rgb, 255:red, 74; green, 144; blue, 226 }  ,fill opacity=1 ] (523.5,94.5) -- (569,94.5) -- (569,102) -- (523.5,102) -- cycle ;
        \draw  [fill={rgb, 255:red, 74; green, 144; blue, 226 }  ,fill opacity=1 ] (570.5,85.5) -- (616,85.5) -- (616,93) -- (570.5,93) -- cycle ;
        \draw  [fill={rgb, 255:red, 74; green, 144; blue, 226 }  ,fill opacity=1 ] (570.5,68.5) -- (616,68.5) -- (616,76) -- (570.5,76) -- cycle ;
        
        \draw (82.73,24.3) node [anchor=north west][inner sep=0.75pt]   [align=left] {$\displaystyle X_{2}$};
        \draw (37.98,24.05) node [anchor=north west][inner sep=0.75pt]   [align=left] {$ $};
        \draw (132.93,24.55) node [anchor=north west][inner sep=0.75pt]   [align=left] {$\displaystyle X_{3}$};
        \draw (34.24,22.7) node [anchor=north west][inner sep=0.75pt]   [align=left] {$\displaystyle X_{1}$};
        \draw (34.24,109.97) node [anchor=north west][inner sep=0.75pt]   [align=left] {$\displaystyle X_{4}$};
        \draw (132.18,109.97) node [anchor=north west][inner sep=0.75pt]   [align=left] {$\displaystyle X_{5}$};
        \draw (261.73,24.3) node [anchor=north west][inner sep=0.75pt]   [align=left] {$\displaystyle X_{2}$};
        \draw (216.98,24.05) node [anchor=north west][inner sep=0.75pt]   [align=left] {$ $};
        \draw (311.93,24.55) node [anchor=north west][inner sep=0.75pt]   [align=left] {$\displaystyle X_{3}$};
        \draw (213.24,22.7) node [anchor=north west][inner sep=0.75pt]   [align=left] {$\displaystyle X_{1}$};
        \draw (213.24,109.97) node [anchor=north west][inner sep=0.75pt]   [align=left] {$\displaystyle X_{4}$};
        \draw (311.18,109.97) node [anchor=north west][inner sep=0.75pt]   [align=left] {$\displaystyle X_{5}$};
        \draw (176,55.4) node [anchor=north west][inner sep=0.75pt]    {$\Delta t$};
        \draw (88,137.4) node [anchor=north west][inner sep=0.75pt]    {$t_{0}$};
        \draw (270,137.4) node [anchor=north west][inner sep=0.75pt]    {$t_{1}$};
        \draw (490,137.4) node [anchor=north west][inner sep=0.75pt]    {$t_{0}$};
        \draw (611,137.4) node [anchor=north west][inner sep=0.75pt]    {$t_{1}$};
        \draw (384,46.4) node [anchor=north west][inner sep=0.75pt]    {$\pi _{0}(\mathscr{G}_{\mathcal{X}})$};
        \draw (543.5,21.9) node [anchor=north west][inner sep=0.75pt]    {$H_{0}$};
        \draw (172.5,141.9) node [anchor=north west][inner sep=0.75pt]    {$\mathscr{G}_{\mathcal{X}}$};
    \end{tikzpicture}
    \caption{Two snapshots of a DBG (left) at times $t_0$ and $t_1$, and the resulting $H_0$ barcode (right). At $t_0$, the graph is fully connected; by $t_1$, it has split into two clusters, creating a new bar in the barcode.}
    \label{fig:placeholder}
\end{figure}
In Figure \ref{fig:placeholder}, our clustering barcode $H_0$ indicates the persistence of specific clusters between times $t_0$ and $t_1$. As we can see, at time $t_0$, only one cluster exists within the entire DBG. However, after some time, we end up with two different clusters. The barcodes allow us to have clustering information stored continuously across time. In addition to this, SDPH allows for the creation of stable barcodes. We can remember the stability result for DBGs, where a particular DBG is mapped into a formigram in a way that maintains stability. Or, in other words, our formigrams are resistant to noise. This detail is essential in situations where data is entirely precise. And in  addition to this, we can remember that $d_B(\textnormal{dgm}(S_\varepsilon \pi_0(\mathscr{G}_\mathcal X)), \textnormal{dgm}(\pi_0(\mathscr{G}_\mathcal X))) \leq \varepsilon$ for any $\varepsilon > 0$ and a particular $\mathscr{G}_\mathcal X$. Since the \textit{time-interlevel smoothing} of our DBG removes large fluctuations and instability within our DBG $\mathscr{G}_\mathcal X$, it is important to note that the barcodes between the smoothed and the unsmoothed version remain $\varepsilon$ close.

\section*{Acknowledgments}

The first-named author gladly acknowledges the College of Arts and Sciences at Loyola University Chicago and the Department of Mathematics and Statistics for their support in funding and providing an intellectually stimulating environment for research. More specifically, the first-named author is grateful for funding from the Rataj Family Scholarship and the Mulcahy Scholarship.
The second-named author gratefully acknowledges support from the College of Arts and Sciences at Loyola University Chicago for summer research funding, as well as the AMS and Simons Foundation for support through an AMS-Simons PUI grant.
\vspace{-10pt}

\bibliography{refs}

\Addresses

\end{document}